\documentclass{base_class}

\usepackage{ytableau}

\usepackage[size=tiny]{todonotes}

\DeclareMathOperator{\MVal}{MVal}
\newcommand{\mres}{\mathbin{\vrule height 1.6ex depth 0pt width
		0.13ex\vrule height 0.13ex depth 0pt width 1.3ex}}
	
\renewcommand{\FF}{F}

\title{Valuations in  affine convex geometry}

\author[Jakob Henkel]{Jakob Henkel}
\author[Thomas Wannerer]{Thomas Wannerer}

\address{Friedrich-Schiller-Universit\"at Jena, Fakult\"at f\"ur Mathematik und Informatik, Institut f\"ur Mathematik, Ernst-Abbe-Platz 2, 07743 Jena, Germany}
\email{jakob.henkel@gmail.com}
\email{thomas.wannerer@uni-jena.de}
\thanks{TW was supported by DFG grant WA 3510/3-1.}
\subjclass[2020]{52A39, 52B45, 52A40}

\begin{document}

\begin{abstract}
	In convex geometry, the constructions that assign to a convex body its difference body,  projection body, or volume have the following properties: They are (1) invariant under volume-preserving linear changes of coordinates; (2) continuous; (3) finitely additive, and the resulting convex bodies are subsets of an irreducible representation of the special linear group. In this paper we explore the question whether there exist other constructions with these properties. We discover a surprising dichotomy: There are no new examples if one assumes translation invariance, but a plethora of examples without this assumption. 
\end{abstract}

\maketitle

\section{Introduction}

\subsection{General background}
Constructions  of convex bodies that are invariant under volume-preserving linear changes of coordinates play an important role in convex geometry. Prominent examples of this type are   the polar body, the difference body, the centroid body, the projection body,  the intersection body   or one of the many ellipsoids associated to convex bodies such as the L\"owner or John ellipsoids, see, e.g., the books by Schneider \cite{Schneider:BM}, Artstein-Avidan, Giannopolos, and Milman \cite{AGM:AGA1,AGM:AGA2}, and Gardner \cite{Gardner}.

Given a convex body $K$ in $\RR^n$, i.e.,  a non-empty convex compact set, the aforementioned constructions yield again convex bodies in $\RR^n$. However, not fixing coordinates  from the start and working instead with convex bodies in a finite-dimensional real vector space $V$, reveals  that the resulting convex bodies lie either in $V$  or its dual $V^*$. There are also important invariant constructions that yield convex bodies in vector spaces of dimension higher than that of  $V$. Viewing origin-symmetric and full-dimensional convex bodies in $V$ as unit balls of norms, one sees that the operator norm on $\End(V)\simeq V^*\otimes V$ is one such example. 

The purpose of this article is to systemically investigate constructions of convex bodies invariant under volume-preserving linear changes of coordinates subject to two additional conditions. Let us first fix our notation. We denote by $\NN$  the non-negative integers. In the following $V$ denotes a finite-dimensional real vector space of dimension $n\geq 2$. 
Let $\calK(V)$ denote the space of convex bodies in $V$ with the topology induced by the Hausdorff metric and let $\calK_0(V)$ be the subspace of convex bodies containing the origin  in their interior. 
Let $\rho \colon \SL(V)\to \GL(W)$ be a Lie group representation of the special linear group on a finite-dimensional real vector space $W$.  We denote by $D(V)$ the vector space of densities, i.e.\ scalar multiples of the Lebesgue measure, on $V$.
In this article we  consider maps $\Phi$ defined either on all of $\calK(V)$ or only on the subspace $\calK_0(V)$  and taking values in $\calK(W)$ that are
 \begin{enuma}
\item invariant:  $\rho(T) \Phi(T^{-1} K) = \Phi K$
  for all $T\in \SL(V)$ and all $K$;
\item continuous;
\item Minkowski valuations: 
$$\Phi(K\cup L)+ \Phi(K\cap L)= \Phi K + \Phi L $$
 whenever $K\cup L$ is again convex.
\end{enuma}

Here $K+L$ denotes the Minkowski sum of convex bodies. In convex geometry, a map defined on convex bodies taking values in an abelian semigroup that satisfies property (c) is  called a valuation. In the particular case that the semigroup is $\calK(W)$ with the Minkowski addition one speaks of Minkowski valuations. Since the affine surface area of a convex body, a central concept in the affine geometry of convex bodies, is not continuous, but merely upper semicontinuous, we will relax condition (b) below to upper semicontinuity.

Since the  volume  and the volume of projections of convex bodies are finitely additive, they are  important sources for constructions of valuations. Of the examples mentioned at the beginning, the difference body, the   projection body, and the centroid body (the moment body) are Minkowski valuations. 
This list of examples is not arbitrary. In fact, the following striking theorem of Ludwig demonstrates that the first two are the only possibilities if we impose the additional property of translation invariance and require that all  convex bodies be in $V$ or $V^*$.

To state Ludwig's theorem, let $\Delta K= K+(-K)$ denote the difference body of $K$. Here $-K$ denotes the reflection of $K$ in the origin. 
To define the projection body of $K$, we have to fix a  positive density $\vol$ on $V$.
In terms  of its support function on $V^{**}\simeq V$,  the projection body  of $K$ is the convex body in $V^*$  defined 
by 
$$ h_{\Pi K} (x)=\dt \vol( K+ t[0,x] ), \quad x\in V,$$ 
where $[0,x]$ denotes the line segment with end points  $0$ and $x$.

In the following we will  always consider $V$ and $V^*$ with their canonical $\SL(V)$-representations.

\begin{theorem}[{\cite{Ludwig:Minkowski}}]\label{thm:Ludwig}
	Let $\Phi\colon \calK(V)\to \calK(W)$  be an invariant continuous Minkowski valuation and assume in addition that $\Phi$ is translation-invariant. 
	\begin{enuma}
		\item If $W=V$, then there is a constant $c\geq 0$ such that $\Phi = c\Delta$.
		\item If $W = V^*$, then there is a constant $c\geq 0$ such that $\Phi = c\Pi$.
	\end{enuma}

\end{theorem}

 As an immediate consequence of  Hadwiger's characterization of  rigid motion invariant valuations, one may add another case to Ludwig's theorem:
\begin{enuma}
\setcounter{enumi}{2}
\item {\it If $W= \RR$, then there exist line segments $I_0,I_1\subset \RR$ such that $\Phi K= I_0 + \vol(K) I_1 $.} 
\end{enuma}

\subsection{Our results}

Our main result shows that if we insist on translation invariance, then  there are no invariant continuous Minkowski valuations beyond those covered by Ludwig's theorem.

We call a Minkowski valuation $\Phi$ non-trivial, if there exists a convex body $K$ such that $\Phi K \neq \{0\}$. 
A representation $W$ is called irreducible, if it does not possess  non-trivial invariant subspaces. 
\begin{theorem}\label{thm:mainA} 
	Let $\Phi\colon \calK(V)\to \calK(W)$    be a non-trivial invariant continuous Minkowski valuation. If $W$ is irreducible and $\Phi$ is translation-invariant, then $W$ is isomorphic to either $\RR$, $V$, or $V^*$. 
		
\end{theorem}

 Since irreducible representations are the building blocks of all representations, Theorem~\ref{thm:mainA} provides information also if $W$ is not irreducible. Indeed, the following corollary says that the most general situation is really just  $W\simeq V^p= V\oplus \cdots \oplus V$ ($p$ summands)  and $W\simeq (V^*)^p$. In the following let $\vol$ be a fixed positive density on $V$.

\begin{corollary}\label{cor:reducible}
	Let $\Phi\colon \calK(V)\to \calK(W)$ be an invariant, continuous Minkowski valuation. If $\Phi$ is translation-invariant, then there exist
	\begin{itemize}
		\item an invariant subspace  of $W$ isomorphic to  $\RR^p \oplus V^q \oplus (V^*)^r$  for certain numbers $p,q,r\in \NN$,
		\item convex bodies $L_0, L_n\subset \RR^p$, and 
		\item  
		translation-invariant, $\SL(V)$-invariant, continuous Minkowski valuations $\Phi_1\colon \calK(V)\to \calK(V^q)$ and $\Phi_{n-1}\colon \calK(V)\to \calK((V^*)^r)$, which  are  homogeneous of degree $1$ and $n-1$, 
	\end{itemize}
	such that 
	$$ \Phi(K)= L_0 + \Phi_1(K)+ \Phi_{n-1}(K)+ \vol(K) L_n,$$
	for all convex bodies $K\in \calK(V)$.  	
\end{corollary}
The Minkowski valuations $\Phi_1$ and $\Phi_{n-1}$ are closely related to the difference and projection body. We discuss examples of such valuations in Section~\ref{sec:reducible}.

 Theorem~\ref{thm:mainA} is all the more surprising if one considers the  plethora of examples of invariant Minkowski valuations we find without the condition of translation invariance. 
  As we are going to discuss in Section~\ref{sec:Lp}, the boundary cases $p=0$ or $q=0$ of the following theorem are closely related to the $L^p$ centroid and projections bodies introduced by Lutwak, Yang, and Zhang \cite{LYZ:LpAffine}.  For $p,q>0$, however, distinctive new features arise, see Section~\ref{sec:11}. In connection with tensor valuations, an analogous construction was considered  by Haberl and Parapatits \cite[Eq. (2)]{HaberlParapatits:CentroAffine}.

Let $\nu_K(x)$ and $\kappa_K(x)$ denote the outer unit normal and the Gauss curvature at  $x$ of a convex body $K$ in $\RR^n$. We denote by $\Sym^p V$ the $p$th symmetric power of $V$. If $V$ is euclidean, then  $\Sym^p V$ can be equipped with a natural euclidean inner product, see Section~\ref{sec:symPowers} below.

\begin{theorem} \label{thm:newMink}
	Let $p,q$ be non-negative integers and let $\vol$ be a positive density on $V$. There exist  invariant continuous Minkowski valuations 
	$$\Phi^{p,q}\colon \calK_0(V)\to \calK(\Sym^p V \otimes \Sym^q V^*)$$
	and 
	$$ \Psi^{p,q}\colon \calK_0(V)\to \calK(\Sym^p V^* \otimes \Sym^q V)$$ such that, after a choice of euclidean inner product on $V$ with the  euclidean volume  equal to $\vol$, for all convex bodies 
	$K$  with smooth and strictly positively curved boundary
	\begin{equation}\label{eq:PhipqK} h_{\Phi^{p,q}(K)}(\phi) =   \int_{\partial K }  |\langle \phi, x^p \otimes \nu_K(x)^q\rangle | \langle x, \nu_K(x)\rangle^{1-q} dx  \end{equation}
	for $\phi\in\Sym^p V^* \otimes \Sym^q V$ and 
	$$ h_{\Psi^{p,q}(K)}(\psi) =   \int_{\partial K }  |\langle \psi, \nu_K(x)^p \otimes x^q\rangle | \langle x, \nu_K(x)\rangle^{-(n+p)} \kappa_K(x)  dx$$
	for $\psi\in\Sym^p V \otimes \Sym^q V^*$.
\end{theorem}

The  family of valuations thus defined is closed under taking polarity, 
 see Proposition~\ref{prop:polarityPhiPsi} for a precise statement.

The  representations $\Sym^p V \otimes \Sym^q V^*$ are in general  not irreducible. Hence decomposing them into irreducible subspaces we  obtain further examples of non-trivial invariant Minkowski valuations. By determining the highest weights of the representations that arise in this fashion (see Section~\ref{sec: weight spaces} for a summary of highest weight theory), we obtain the following result.

\begin{theorem}\label{thm:hwt} \label{thm:TensorDecomp}
	If $W$ is an irreducible representation of $\mathrm{SL}(V)$ with highest weight $$(p+q) \varepsilon_1 + q (\varepsilon_2 + \dots + \varepsilon_{n-1}),$$
	where $p,q \in \NN$, then there exists a non-trivial invariant continuous Minkowski valuation $\calK_0(V) \to \calK(W)$.
\end{theorem}

Whether the  condition on the highest weight specified in Theorem~\ref{thm:TensorDecomp} is not only sufficient, but also  necessary  for the existence of non-trivial invariant continuous  Minkowski valuations $\calK_0(V) \to \calK(W)$ remains an open problem. 
In low dimensions,   Theorem~\ref{thm:TensorDecomp} places no restrictions on the highest weights:

\begin{corollary}\label{cor:dimLeq3} Let the dimension of $V$ be at most $3$.  If $W$ is any representation of $\SL(V)$, then there exists a non-trivial invariant continuous Minkowski valuation $\Phi\colon \calK_0(V)\to \calK(W)$. 	
\end{corollary}

The affine surface area of convex bodies is a fundamental invariant in the affine geometry of convex bodies. 
 It arises naturally when approximating the volume of a convex body by the volume of random polytopes, see, e.g., the recent survey by Sch\"utt and Werner \cite{SchuttWerner:SurveyASA} for this and related interesting results. 
  Compared with the volume and the usual surface area of convex bodies, the affine surface area  has two seemingly   pathological properties: It vanishes on polytopes and is not continuous, but only upper semicontinuous. The latter property  means that 
$$ \limsup_{j\to \infty} \Omega(K_j)\leq \Omega(K)$$
holds for every convergent sequence of convex bodies $K_j \to K$ and was established by Lutwak \cite{Lutwak:Asa}. 
The special role of affine surface is also highlighted by a theorem of Ludwig and Reitzner~\cite{LudwigReitzner:AffineSurface}: The cone of upper semicontinuous real-valued valuations invariant under volume-preserving affine transformations is three-dimensional and  spanned by the constant valuation, volume, and affine surface area.

Since a convex function is twice differentiable almost everywhere, 
at almost every boundary point $x$ of a convex body $K$ in $\RR^n$ with non-empty interior there exist a unique outward-pointing unit  normal $\nu_K(x)$
and the  Gauss curvature $\kappa_K(x)$.  For a convex body $K$ containing the origin in its interior, let
$$\wt \kappa_K(x)= \frac{\kappa_K(x)}{\langle x,\nu_K(x)\rangle^{n+1} }$$
denote the Gauss curvature divided by a suitable power of the support function of $K$.  This expression is well known to be  invariant under volume-preserving linear transformations, see  Section~\ref{sec:GaussAffine} for a computation-free explanation of this fact.

Let $\mathrm{Conc}(0,\infty)$ be the set of concave functions $\phi\colon [0,\infty)\to [0,\infty)$ with the properties that 
$\phi(0)=0$, $\lim_{t\to 0} \phi(t)= 0$, and $\lim_{t\to \infty} \phi(t)/t=0$. 
Inspired by the work of Ludwig and Reitzner \cite{Ludwig:GeneralAsa,LudwigReitzner:Classification} on general affine surface areas,  we construct  two families  of invariant upper semicontinuous Minkowski  valuations. We call a map $\Phi\colon \calK_0(V)\to \calK(W)$ upper semicontinuous if 
for every convex body $D$ containing the origin in the interior and every convergent sequence $K_j\to K$ of convex bodies there exists an index $j_0\in \NN$ such that 
\begin{equation}\label{eq:upper semi} \Phi K_j \subset   \Phi K + D\end{equation}
holds for all $j\geq j_0$.

Let $\vol$ be a positive density on $V$. In Section~\ref{sec:GaussAffine} we show how to define $\wt \kappa_K(x)$ using only this density. Associated to $\vol$ is also a Borel measure supported on $\partial K$,  the cone volume measure,
$$ \upsilon_K(U) = \vol(\{ tx\colon x\in U,\ 0\leq t\leq 1\}), \quad  U\subset \partial K.$$

\begin{theorem}\label{thm:semi}
	Let  $\vol$ be a positive density on $V$, let $p$ be a non-negative integer, and let $f\in \mathrm{Conc}(0,\infty)$. Then  $\Phi_f^p\colon \calK_0(V) \to \calK(\Sym^p V)$, 
	\begin{equation}\label{eq:Phifp} h_{\Phi_f^p K}(\phi) = \int_{\partial K} |\langle \phi,x^p  \rangle |       f (\wt \kappa_K(x))\; d\upsilon_K(x),\end{equation}
	and $\Psi_f^p\colon  \calK_0(V) \to \calK(\Sym^p V^*)$,
	\begin{equation} h_{\Psi_f^p K}(\psi) = \int_{\partial K} |\langle \psi,\xi^p  \rangle | \langle x, \xi \rangle ^{-p}     f (\wt \kappa_K(x)) \; d\upsilon_K(x),\end{equation}
	where $\xi\in V^*$ is an outward-pointing conormal of $K$ at $x$,
	are  well-defined invariant upper semicontinuous Minkowski valuations.	
\end{theorem}

The family of Minkowski valuations thus defined turns out to be closed under taking polarity, see Proposition~\ref{prop:polarity}.

\subsection{Comparison with other work}

Our interest in invariant Minkowski valuations was sparked in part by the recent proof of the affine quermassintegral inequalities by Milman and Yehudayoff \cite{MilmanYehudayoff:Affine}. One stumbling block there  was to construct a convex body---analogously to the difference and the projection body---for projections onto subspaces of dimension strictly between $1$ and $n-1$. Such a convex body would be expected to be a subset of the exterior power $\largewedge^k V$, which is an irreducible representation of the special linear group (recall also that $\largewedge^1 V=V$ and $\largewedge^{n-1} V\simeq V^*$). Milman and Yehudayoff found a work-around for this issue and proved the affine quermassintegral inequalities, but the question whether there is such a construction remained open. 

The term Minkowski valuation was introduced by Ludwig in the seminal paper \cite{Ludwig:Minkowski},   building on earlier work of Schneider~\cite{Schneider:Endo}. The research on Minkowski valuations so far falls  into three broad categories:  Improvements of the theorems by  Ludwig \cite{Ludwig:Minkowski,Ludwig:Projection,Ludwig:Intersection}  by weakening  or changing the assumptions,   for example  by removing translation invariance or  continuity \cite{Parapatits:ContraLp,Parapatits:CovariantLp,AbardiaEtal:VC,Haberl:Minkowski,LiLeng:LpMinkowski,LengMa:Plane}.
Papers in the second category require invariance  not under the action of the full special linear group,  but only of subgroups such as the complex special linear  or the orthogonal groups \cite{AbardiaBernig:Complex,Haberl:ComplexAffine,Dorrek:Endomorphisms,SchusterW:Generalized,OrtegaSchuster:FixedPoints}.
The third category consists of investigations that consider valuations on different classes of subsets  or functions on  $V$, e.g., lattice polytopes or  convex functions,   with values in the space of convex bodies in $V$ or $V^*$ \cite{CLM:Minkowski,BoroczkyLudwig:Minkowski}. The idea of systematically investigating Minkowski valuations with values in $\calK(W)$ where $W$ is not necessarily isomorphic to  $V$ or $V^*$ appears in this paper for the first time. However, when a first version this work was finished we have learned about a paper of Schneider~\cite{Schneider:GenDiff} and the recent preprints \cite{Haddad_etal:GenProj, Haddad_etal:LpIsoperimetric} where with a different intention generalized difference and  projection bodies taking values in $V^p$ and $(V^*)^p$ are introduced. To the best of our knowledge,  the upper semicontinuous Minkowski valuations of Theorem~\ref{thm:semi} have not been considered  in the literature.

\subsection{Organization}

In Sections 2 to 4 we collect  for later use  results   from convex geometry, representation theory, and the theory of valuations. We prove Theorem~\ref{thm:mainA} in Sections 5 and  6. The  new examples of invariant Minkowski valuations presented in Theorem~\ref{thm:newMink} are constructed and discussed  in Section 8. In Section 9 we prove Theorem~\ref{thm:semi} on upper semicontinuous Minkowski valuations. In Section 10 we discuss open problems arising from this work.

\subsection{Acknowledgments} We thank Dmitry Faifman and Monika Ludwig for their helpful comments on  earlier versions of this work.

\section{Convex geometry}
\label{sec:convexity}

We collect in this section for later reference several facts about the boundary of convex bodies.
With the exception of the invariant definition of the support function that we discuss in the following paragraph, we refer the reader for all other elementary concepts from convex geometry not defined here  to the book by Schneider~\cite{Schneider:BM}. 
 
Let $V$ be a finite-dimensional real vector space.
The  support function of a convex body $K$ in $ V$ is the  sublinear function 
$$ h_K(\xi)= h(K,\xi)= \max_{x\in K} \langle x,\xi\rangle, \quad \xi\in V^*.$$
The assignment $K\mapsto h_K$  is a bijection between  convex bodies in $V$ and sublinear functions on $V^*$.

By a classical theorem of Alexandrov, a convex function is twice differentiable at almost every point (see, e.g., \cite[Theorem 6.4.1]{EvansGariepy}).   This theorem has  important implications for the boundary structure of convex bodies. Indeed, locally around a boundary point $x$, the boundary of a convex body with non-empty interior can be represented as the graph of a convex function.  If this function is twice differentiable at $x$, then $x$ is called a normal boundary point.  Consequently, if $K\subset \RR^n$ is a convex body with non-empty interior, then   almost all boundary points (with respect to the $(d-1)$-dimensional Hausdorff measure) are normal.

At boundary points $x$ of $K$ where there exists a unique outer  normal $u=\nu_K(x)$ (which is the same as requiring that there is a unique supporting hyperplane through $x$) define $\Delta=\Delta(K,x,\delta)$ to be the number such that 
the slice $$ \{ y\in K \colon h_K(u)-\Delta\leq\langle y,u\rangle\leq h_K(u)\}.$$
has volume $\delta$. 
By \cite[Hilfsatz 2]{Leichtweisz:ZurAffinoberflaeche} (see also \cite[Lemma 10]{SchuttWerner:Floating}), if $x$ is normal, the limit
$$ c_n \lim_{\delta\to 0} \frac{ \Delta(K,x,\delta)}{\delta^{2/(n+1)}},$$ 
 where $c_n$ is a dimensional constant, exists and coincides  with the  $(n+1)$-st root of  the Gauss curvature $\kappa_K(x)$ of $K$ at $x$.

Given a convex body $K$ in $\RR^n$ and $r>0$ let  $$(\partial K)_r=\{ x\in \partial K \colon \exists a \colon x\in B(a,r)\subset K\}$$
and $(\partial K)_+ = \bigcup_{r>0}  (\partial K)_r$.
Each boundary point $x\in (\partial K)_+$ possesses a unique outer unit normal. 
\begin{lemma}[{\cite[Lemmas 2.1, 2.2, and 2.3]{Hug:Contributions}}]
	\label{lemma:Gauss Lipschitz}
Let $K\subset \RR^n$ be a convex body with non-empty interior und let $r>0$. The following properties hold:
\begin{enuma}
	\item $\partial K \setminus (\partial K)_+$  has $(n-1)$-dimensional Hausdorff measures zero;
	\item the restriction of the Gauss map $\nu_K$ to $(\partial K)_r$ is Lipschitz;
	\item  at a.e.\ $x\in  (\partial K)_r$ the  Jacobian of the Gauss map equals $\kappa_K(x)$. 
\end{enuma}
\end{lemma}

At points $u\in S^{n-1}$ where the support function of $K$ is twice differentiable,  denote by $\alpha_K(u)$ the sum of the principal $(n-1)$-minors of the Hessian of the support function of $K$ at $u$. In terms of the mixed discriminant,  $\alpha_K(u)$ can be expressed as 
$$ \alpha_K(u) =  D( u\otimes u , \nabla^2 h_K(u), \ldots, \nabla^2 h_K(u)).$$

\begin{lemma}[{\cite[Lemma 2.7]{Hug:Contributions}}]\label{lemma:K+}
	Let $K$  be a convex  body in $\RR^n$. Suppose that its support function is  twice differentiable at $u\in S^{n-1}$. Let $x=\nabla h_K(u)$. Then the following statements are equivalent: 
	\begin{enuma}
		\item $x\in (\partial K)_+$. 
		\item $\alpha_K(u)>0$. 
	\end{enuma}
	
\end{lemma}

The following two results relate $\kappa_K$ and $\alpha_K$. 

\begin{lemma}[{\cite[Lemma 2.6]{Hug:Contributions}}] \label{lemma:alphakappa}Let $K\subset \RR^n$ be a convex body with non-empty interior. 
	Then for a.e.\ $u\in \nu_K((\partial K)_+)$, the support function $h_K$ is twice differentiable  at $u$, $x=\nabla h_K(u)$ is a normal boundary point of $K$, and 
$$\kappa_K(x) \alpha_K(u)=1.$$
\end{lemma}

We denote by $K^*$ the polar body of $K$.

\begin{theorem}[{\cite[Theorem 2.2]{Hug:CurvRel}}] \label{thm:curvRel} Let $K\subset \RR^n$ be a convex body containing the origin in its interior. For a.e.\ $x\in \partial K$ the following properties hold:
	\begin{enuma} 
		\item $x$ is a normal boundary point of $K$;
		\item the support function $h_{K^*}$ is second order differentiable at $x$;
		\item $ \kappa_K(x) = \langle x/|x|, \nu_K(x)\rangle^{n+1} \alpha_{K^*}( x/|x|).$
\end{enuma}
\end{theorem}

\subsection{Gauss curvature and affine geometry}\label{sec:GaussAffine}

Let us recall how affine constructions give rise to expressions involving the Gauss curvature of convex bodies. Since we will only need Corollary~\ref{cor:kappaAlpha} in the following,  readers not interested in these constructions may skip them.

Let $V$ be a finite-dimensional real vector space, let $M$ be a smooth manifold, and let $f\colon M\to V$ be a smooth immersion. The second fundamental form, defined in classical differential geometry using a euclidean inner product on $V$, is in fact an affine concept.  As explained for example in \cite{Gromov:SignMeaning}, let $x\in M$, $U= df_x(T_xM)$ and let $\pi\colon V\to V/U$ denote the canonical projection. Since $x$ is a critical point of $\pi\circ f$, its Hessian  
$$h = \nabla^2 (\pi \circ f)\in (T_x^*M)^{\otimes 2}\otimes (V/U)$$ at $x$ does not depend on the choice of connection on $M$. Choosing a euclidean inner product on $V$ to identify $V/U=U^\perp$, $h$ becomes the usual second fundamental form. 

If $M$ is a hypersurface, then the determinant of $h$ is an element of 
$$  (\largewedge^{n-1} T^*_x M)^{\otimes 2} \otimes (V/U)^{\otimes (n-1)}.$$
Recall that if $0 \to U\to V\to W\to 0$ is exact and $\dim U=m$, $\dim V=n$, then there is a canonical isomorphism
$$ \largewedge^n V^*  \otimes \largewedge^m  U \simeq \largewedge^{n-m} W^*.$$ 
Applied to $0\to T_xM \to V\to V/U\to 0$, we find that 
\begin{equation}\label{eq:deth}\det h \in   (\largewedge^{n} V^*)^{\otimes 2} \otimes (V/U)^{\otimes (n+1)}.\end{equation}

Let $or(V)$ denote the $1$-dimensional vector space of functions 
$$ f\colon \{ (v_1,\ldots, v_n)\in V^n\colon  v_1,\ldots, v_n \text{ linearly independent}\}\to \RR$$ with the property
$$ f(Tv_1,\ldots, T v_n)= \operatorname{sign}( \det T) f(v_1, \ldots, v_n)$$ 
for all $T\in \GL(V)$. Observe that there exist canonical isomorphisms 
$$ or(V)\otimes or(V)\simeq \RR$$
and 
$$ D(V) \simeq or (V) \otimes \largewedge^n V^*,$$
where $D(V)$ denotes the space of densities on $V$,  see e.g.\ \cite{Alesker:Fourier}. Hence $(\largewedge^{n} V^*)^{\otimes 2}$ and $D(V)^{\otimes 2}$ are canonically isomorphic. 
Thus, if we fix a density $\vol$ on $V$, we see that $\det h$ of \eqref{eq:deth} can be defined as an element of $ (V/U)^{\otimes (n+1)}$.

Let us specialize the above now to $M=\partial K$, where $K\subset V$ is a convex body with smooth boundary containing the origin in the interior. In this case there  always exists a canonical element in $V/U$, namely $\pi(x)$, and we may identify $ (V/U)^{\otimes (n+1)}$ with $\RR$.  After this identification, $\det h$ becomes a number that we denote by $\wt \kappa_K(x)$. After a choice of inner product such that  $\vol$ coincides with the euclidean density, we can express this quantity in terms of the support function and the usual Gauss curvature as follows:
$$\wt \kappa_K(x) = \langle x, \nu_K(x)\rangle^{-(n+1)} \kappa_K(x).$$
By construction, $\wt\kappa_K$ is invariant under  volume-preserving linear transformations of $K$. 

Also for general convex bodies $K$ containing the origin in the interior there exists an affine invariant construction of $\wt \kappa$. Indeed, instead of $\Delta(K,x,\delta)$, let  for every outward-pointing conormal $\xi$ at $x$ now  $\Delta=\Delta(K,x,\xi,\delta)$ be the number so that the volume of the slab 
$$\{y\in K\colon \langle x,\xi\rangle -\Delta \leq \langle y,\xi\rangle \leq \langle x,\xi\rangle\}$$
is $\delta$. Passing to the limit as before, multiplying by $c_n$, and raising the result to the $(n+1)$-st power, we obtain a number $\kappa_K(x,\xi)$. Clearly, $\kappa_K(x,t \xi) =  t^{n+1} \kappa_K(x, \xi)$ for all $t>0$. Thus, 
$$ \wt \kappa_K (x)= \langle x, \xi\rangle^{-(n+1)} \kappa_K(x,\xi)$$
is well-defined for a.e.\ $x\in \partial K$ and by  construction invariant under  volume-preserving linear transformations.

Let us next define an affine version of $\alpha_K$. As in the discussion above, we view the determinant of a symmetric bilinear form $q\colon V^*\times V^*\to \RR$ and more generally the mixed discriminant of $q_1,\ldots,q_n$  as elements of $(\largewedge^n V)^{\otimes 2}$. 
At points $\xi\neq 0$ where $h_K$ is twice differentiable, the Hessian $\nabla^2 h_K(\xi)$ is always a degenerate bilinear form. Let $x\in K$ be the unique element in $K$ with $\langle x,\xi\rangle = h_K(\xi)$. Define the number
$$ \wt \alpha_K(\xi)= h_K(\xi)^{n-1} D(x\otimes x, \nabla^2 h_K(\xi),\ldots,  \nabla^2 h_K(\xi))$$
after identifying $(\largewedge^n V)^{\otimes 2}\simeq \RR$ using the fixed density $\vol$. 
Thus defined, the function $\wt \alpha\colon V^*\setminus\{0\}\to \RR$ is $0$-homogeneous and clearly invariant under volume-preserving linear transformations of $K$. 
After a choice of inner product such that $\vol$  coincides with  the euclidean density, we can express this affine invariant as
$$ \wt\alpha_K(u) = h_K(u)^{n+1} \cdot \alpha_K(u), \quad u\in S^{n-1}.$$

Let us close this section with the following useful fact. 
\begin{corollary}\label{cor:kappaAlpha}
	Let $K\subset \RR^n$ be a convex body containing the origin in its interior. Then  
	$$ \wt \kappa_K(x) = \wt \alpha_{K^*}(x)$$
	for a.e.\ $x\in \partial K$.
\end{corollary}
\begin{proof}This follows from  Theorem~\ref{thm:curvRel} using the fact that $|x|= \rho_K(x/|x|) = h_{K^*}(x/|x|)^{-1}$ holds for every $x\in \partial K$,  where $\rho_K$ denotes the radial function of $K$. 
\end{proof}

\subsection{Conormal cycle}

Let $\PP_+(V^*)$ denote  the set of rays in $V^*$ emanating from the origin. We can think of a ray as an equivalence class $[\xi]$ for the equivalence relation in $V^*\setminus\{0\}$ defined by $\xi\sim \eta$, if there exists $\lambda>0$ such that $\xi= \lambda \eta$. 

Sufficiently tame closed subsets $A$  of $V$ and   in particular convex bodies  define after a choice of orientation on $V$  an integral current in $\PP_V=V\times \PP_+(V^*)$, the conormal cycle $\nc(A)$. 
As a set, the conormal cycle of convex bodies consists of all pairs $(x,[\xi])$ such that 
$x\in K$ and 
$$ \langle \xi,y-x\rangle \leq 0$$
for all $y\in K$. 

We collect the properties of the conormal cycle relevant to this paper in the following proposition and refer the reader for more information to  \cite{Fu:IGRegularity}.

\begin{proposition}[{\cite{Alesker:VMfdsIII}}]\label{prop:nc} Let $\omega\in or(V)\otimes \Omega^{n-1}(\PP_V)$ be a smooth differential $(n-1)$-form with values in $or(V)$ and let $K$ denote a convex body in $V$.
	The conormal cycle of a convex body has the following properties.
	\begin{enuma}

	\item  The function 
	$$ \phi(K)= \int_{\nc(K)} \omega$$ is a  continuous valuation  and is independent of the choice of orientation of $V$ used to define the conormal cycle.
	\item For fixed $K$, the function
		$$ C_K(U)= \int_{\nc(K)\cap \pi_1^{-1}(U)} \omega,$$
		where $\pi_1\colon \PP_V\to V$ is the projection to the first factor and  $U\subset V$ is a Borel set, is a finite Borel measure. If $K_j\to K$, then $C_{K_j}\to C_K$ weakly. 
	\item As currents, 
	 $$\nc(T K) = \operatorname{sign}(\det T) \cdot T_{*} \nc(K)$$ for every $T\in \GL(V)$, where $T(x,[\xi])= (Tx, [T^{-*} \xi])$.

	\end{enuma}

\end{proposition}

The following lemma describes the conormal cycle of the polar body of $K$ in terms of the conormal cycle of $K$. 

\begin{lemma}\label{lem:polar}
	Consider the open subset  $\PP_V^+=\{(x,[\xi] )\in \PP_V\colon \langle x,\xi\rangle >0\}$ and let $F_V\colon \PP_V^+ \to \PP_{V^*}^+$ be 
	defined by $$F_V(x,[\xi])= (\frac{\xi}{\langle x,\xi \rangle } , [x]).$$ Then $F_V$ is a diffeomorphism and  $F_{V^*}\circ F_V=\id$ after the canonical identification of $\PP_{V^{**}}^+$ and $\PP_V^+$.		
	For all convex bodies $K\in \calK_0(V)$	 
	\begin{equation}\label{eq:ncPolar}(F_V)_*\nc(K)= \nc(K^*)\end{equation}
as currents.

\end{lemma}

\begin{proof}
	The analytic properties of $F_V$ are straightforward to verify. If  the convex body $K\in \calK_0(V)$ has a smooth and strictly positively curved boundary, \eqref{eq:ncPolar} can directly verified. The general case follows by approximation and continuity of the conormal cycle.
\end{proof}

\section{Valuations}

Let $V$ be a finite-dimensional real vector space of dimension $n$ and let  $\calK(V)$ denote the set of convex bodies in $V$, i.e., non-empty convex compact subset of $V$. A valuation is a function $\phi\colon \calK(V)\to A$ with values in an abelian semigroup $A$ such that 
$$ \phi(K\cup L) + \phi(K\cap L)= \phi(K)+ \phi(L)$$
holds whenever $K\cup L$ is convex. In this article, the semigroup will be either the complex numbers with the usual addition or $\calK(W)$, where $W$ is another finite-dimensional real vector space and the  addition on $\calK(W)$ is the Minkowski addition of convex bodies. Valuations of the latter type are called Minkowski valuations.  

Identifying $V\simeq \RR^n$ the  set $\calK(V)$ can be given the  topology induced by the Hausdorff metric. Note that this topology does not depend on the identification $V\simeq \RR^n$. Continuity of  valuations with values in $\CC$ or $\calK(W)$ is always understood with respect to this topology. 

A valuation is translation-invariant if $\phi(K+v) = \phi(K)$ holds for all convex bodies $K$ and $v\in V$.

\subsection{Complex-valued valuations}
The theory of translation-invariant and continuous valuations on convex bodies with values in the complex numbers  is  a rich and highly developed subject. We collect here facts from this theory that will be relevant to us. For more information we refer the reader to  Schneider's book \cite[Chapter 6]{Schneider:BM} and the lecture notes by Alesker~\cite{Alesker:Kent}.
	
First of all, let $\Val(V)$ denote the space of complex-valued, translation-invariant, continuous valuations and let $\Val_i(V)$ denote the subspace of $i$-homogeneous valuations, i.e., those valuations satisfying $\phi(tK)= t^i \phi(K)$ for all $t>0$ and all $K$.  McMullen's decomposition theorem is the statement that every $\phi\in \Val(V)$ can be written uniquely as a sum $ \phi= \phi_0+ \cdots + \phi_n$
of homogeneous valuations $\phi_i\in\Val_i(V)$. 

A  density on $V$ defines by restriction to  convex bodies an element of $\Val_n(V)$. Conversely,  a theorem of Hadwiger shows that every element of $\Val_n(V)$ is  
 of this form.

Since every $0$-homogeneous valuation $\phi$  must be constant, the space $\Val_0(V)$ is one-dimensional. 

McMullen's decomposition theorem can be refined by the splitting into even and odd valuations $\Val_i(V)= \Val^+_i(V)\oplus \Val^-_i(V)$. A valuation is called even if $\phi(-K)= \phi(K)$ and odd if $\phi(-K)= -\phi(K)$ holds for all convex bodies $K$. 

The general linear group acts on valuations by
$$ (T\phi)(K)= \phi(T^{-1}K), \quad T\in \GL(V).$$
Clearly, the subspaces $\Val_i^+(V)$ and $\Val_i^-(V)$ are invariant under this action. A deep theorem of Alesker underlying large parts of modern valuation theory is:

\begin{theorem}[{\cite{Alesker:Irreducibility}}]
Under the action of $\GL(V)$ the spaces $\Val_i^+(V)$ and $\Val_i^-(V)$ are irreducible.
\end{theorem}

Here irreducible means that there exist no non-trivial closed invariant subspaces and the topology on $\Val(V)$ is the topology of uniform convergence on compact subsets of $\calK(V)$. 

Even valuations are easier to handle due to the following fact. Let $\phi\in \Val^+(V)$ be homogeneous of degree $i\in \{1,\ldots, n-1\}$. The Klain section $\Klain_\phi$ of $\phi$ is a section of the line  bundle over $\Grass_i(V)$,  the Grassmannian of $i$-dimensional linear subspaces in $V$,  with fiber $\Val_i(E)$ defined by 
$$ \Klain_\phi(E)= \phi|_E \in \Val_i(E), \quad E\in \Grass_i(V).$$
By Hadwiger's characterization of volume mentioned above  $\Val_i(E)=D(E)$ and hence a choice of euclidean inner product trivializes the line bundle. Hence after such a choice, we can speak of the Klain function  of $\phi$. The following fundamental fact was proved by Klain.

\begin{theorem}[{\cite{Klain:Even}}] \label{thm:Klain}
	Let $\phi \in \Val_i^+(V)$. If $\Klain_\phi =0$, then $\phi=0$. 
\end{theorem}

For the proof of Theorem~\ref{thm:mainA} we will need the following consequence of a more precise result due to  Alesker and Bernstein.

\begin{theorem}[{\cite{AleskerBernstein:Range}}] \label{thm:imageKlain}Let $n\geq 4$.  For $k\in \{2,\ldots, n-2\}$ the image of the Klain embedding $\Klain\colon \Val_{i}^+(\RR^n)\to C(\Grass_k(\RR^n))$ is not dense.
\end{theorem}

\subsection{Minkowski valuations}

Given two finite-dimensional vector spaces $V,W$ we denote by $\MVal(V,W)$ the set of translation-invariant and continuous  Minkowski valuations $\calK(V)\to \calK(W)$. Note that this set is not a vector space, but merely a convex cone. 

A Minkowski valuation is called even if $\Phi(-K)= \Phi(K)$ holds for all convex bodies $K$ in $V$; $\Phi$ is called homogeneous of degree $\alpha$ if $\Phi(tK) = t^\alpha \Phi(K)$ holds for all $t>0$ and  all  $K$. If $\Phi$ is continuous and translation-invariant, then by McMullen's decomposition theorem, $\alpha$ is necessarily an integer between $0$ and $n=\dim V$.  We denote by $\MVal_i(V)$ the subset of $i$-homogeneous Minkowski valuations. 

Since $K \mapsto h_{\Phi K}(\eta)$ is a real-valued, continuous valuation for fixed $\eta \in W^*,$ Hadwiger's characterization of $n$-homogeneous valuations yields the following statement about $n$-homogeneous Minkowski valuations. 
\begin{lemma}\label{lemma:nHomMink}  Let $\vol$ denote a positive  
	density  on $V$. Given $\Phi\in \MVal_n(V,W)$, there
	 exists a convex body $L$ in $W$ such that 
	$$ \Phi(K)= \vol(K) L$$
	holds for all $K\in \calK(V)$. 
\end{lemma}

We denote by $\MVal^+_i(V,W)$ the subset of even Minkowski valuations. 
The notion of Klain section can be extended to Minkowski valuations as follows.
Given $\Phi\in \MVal_i^+(V,W)$ we define the Klain section of $\Phi$ as the map $\Klain_\Phi \colon \Grass_i(V)\to \bigsqcup_{E\in \Grass_i(V)} \MVal^+_i(E,W)$, where $\bigsqcup$ stands for disjoint union and   
$$ \Klain_\Phi(E) = \Phi|_E\in \MVal_i^+(E,W).$$
In view of Lemma~\ref{lemma:nHomMink}, given a choice of positive Lebesgue measure $\vol_E$ in each subspace $E$, we may think of the Klain section as a map 
$  \Grass_i(V) \to \calK(W)$.

As before, by considering the real-valued valuations $K\mapsto h_{\Phi K}(\eta)$, $\eta\in W^*$, Theorem~\ref{thm:Klain} can be  immediately generalized to Minkowski valuations.
\begin{proposition}\label{prop:KlainMink}
	Let $\Phi,\Psi\in \MVal_i^+(V,W)$. If $ \Klain_\Phi=\Klain_\Psi$, then $\Phi=\Psi$.
\end{proposition}

\section{Representation theory of the general linear group} \label{sec:repTheory}

The proof of Theorem~\ref{thm:mainA} will also rely on several facts from the representation theory of the general linear group over the real numbers. The  results we need hold in fact more generally for fields of  characteristic zero. Since the restriction to real numbers does not lead to any simplification or greater transparency,  we present these results for a general  field  $\FF$ of characteristic zero.  Our references for this material are the book by Fulton and Harris  \cite{FultonHarris} and the excellent lecture notes by Kraft and Procesi \cite{KraftProcesi}.

\subsection{Rational representations}

 Let $V,W$ be finite-dimensional vector spaces over $\FF$ and let $n=\dim(V).$ We write $\mathrm{End}(W)$ for the space of endomorphisms on $W$. We denote by $\GL(V)\subset \End(V)$ the general linear group and we write $\GL_n(\FF)$ for the group of invertible $n\times n$ matrices with coefficients in $\FF$. 
 
 A function $f \colon \mathrm{GL}(V) \to \FF$ is called polynomial if there is a polynomial  $\tilde{f} \colon  \mathrm{End}(V)\to \FF$ such that $f = \tilde{f}|_{\mathrm{GL}(V)}$;  $f$ is called regular if $\det^r  \cdot f$, where $\det \colon \mathrm{End}(V)\to \FF$ denotes the determinant, is
 polynomial for some $r \in \NN.$

\begin{definition}	A representation $\rho \colon \mathrm{GL}(V) \to \mathrm{GL}(W)$ is called polynomial (resp. rational) if for one---and hence every---basis of $W$ the matrix coefficients $\rho_{ij}$ are polynomial (resp. regular) functions on $\mathrm{GL}(V).$ 
\end{definition}

\begin{example}
 	The standard representation $\rho \colon \mathrm{GL}(V) \to \mathrm{GL}(V)$ is clearly polynomial. Its dual  representation $V^*$ however is rational. Moreover, $\det^r \colon \GL(V)\to \GL_1(\FF)$ is rational for every $r\in \ZZ$. 
\end{example}
\begin{example}
	If $W_1,W_2$ are polynomial (resp.\ rational) representations of $\mathrm{GL}(V),$ then $W_1 \oplus W_2$ and $W_1 \otimes W_2$ are polynomial (resp.\ rational) representations of $\mathrm{GL}(V).$ 
	\end{example}
\begin{example}
 If $W$ is a polynomial (resp.\ rational) representation of $\mathrm{GL}(V),$ then clearly any subrepresentation $U$ is a polynomial (resp.\ rational) representation of $\mathrm{GL}(V).$
\end{example}

Note that there exist natural representations that are not rational.

\begin{example} For $F=\RR$ the  representation $|\det|\colon \GL(V)\to \GL_1(\RR)$ is not rational.
\end{example}

 Here are two important properties of rational representations of the general linear group.
 
 \begin{proposition}[{\cite[Corollary 5.3.2]{KraftProcesi}}]\label{prop:TensorPow}
 	Every irreducible polynomial representation of $\GL(V)$ is isomorphic to a subrepresentation of a unique  tensor power $V^{\otimes m}$. 
 \end{proposition}
A representation $W$ is called completely reducible if it is a sum of irreducible subrepresentation of $W$.

\begin{proposition}[{\cite[Corollary 5.3.1]{KraftProcesi}}] \label{prop:comp_red}
	Every rational representation of $\GL(V)$ is completely reducible.
\end{proposition}

\subsection{Weight spaces}
\label{sec: weight spaces}
As before let $V$ denote a vector space over $\FF$ and let $n = \dim(V).$ In this subsection we fix a basis $b_1 , \dots , b_n$ of $V.$ After this choice we can identify $V$ with $\FF^n.$ Denote by 
$$ T^n := \left\{\begin{pmatrix}
	t_1 & & \\
	& \ddots & \\
	& & t_n
\end{pmatrix} \mid t_i \in \FF\setminus\{0\} \right\} \subset \mathrm{GL}_n(\FF) $$
the $n$-dimensional torus. Further, let $\varepsilon_i \colon T^n\to \GL_1(\FF)$ be the element defined by 
$$\varepsilon_i\left( \begin{pmatrix}
	t_1 & & \\
	& \ddots & \\
	& & t_n
\end{pmatrix} \right) = t_i. $$
The $1$-dimensional rational representations of $T^n$ are all of the form
$$ (p_1\varepsilon_1 + \dots + p_n \varepsilon_n)\left( \begin{pmatrix}
	t_1 & & \\
	& \ddots & \\
	& & t_n
\end{pmatrix} 
\right) = t_1^{p_1}\cdots t_n^{p_n}$$
for $p_i \in \ZZ$. 
These representations constitute  the character group of $T^n$  which we denote by $\mathcal X(T^n).$

The weight space decomposition of a  rational representation $W$ is
$$ W = \bigoplus\limits_{\lambda \in \mathcal X(T^n)} W_\lambda, $$
where
$$ W_\lambda := \{w \in W \mid \rho(t)w = \lambda(t) w, \forall t \in T^n\}. $$
If $W_\lambda \neq 0$, we call $\lambda$ a weight and $W_\lambda $ the corresponding weight space. If $w \in W_\lambda$ is non-zero we say $w$ is a weight vector for $\lambda.$

Denote by $U_n$ the subgroup of  unipotent upper triangular  matrices
$$ \begin{pmatrix}
	1 & \ast & \ast\\
	& \ddots & \ast \\
	&  &1
\end{pmatrix}. $$
The group $U_n$ is generated by the matrices 
$$ u_{ij}(s) = \id + s E_{ij}, \quad j>i, $$
where  $s\in \FF$ and  $E_{ij}$ is the matrix with $1$ at position $(i,j)$ and zeros otherwise. The following lemma describes the action of $U_n$ on weight vectors.

\begin{lemma}[{\cite[§5.7]{KraftProcesi}}]
	\label{lemma: Un action on weight vectors}
	Let $\lambda$ be a weight of $W$ and $w \in W_\lambda$ a weight vector. There are elements $w_k \in W_{\lambda+k(\varepsilon_i - \varepsilon_j)} $ for $k\in \NN,$ where $w_0  = w,$ such that
	$$ \rho(u_{ij}(s))w = \sum\limits_{k\geq 0} s^k w_k, \quad s\in \FF. $$
\end{lemma}

The weights of the form $\sum_{i<j} n_{ij} (\varepsilon_i - \varepsilon_j )$
 where $n_{ij}\geq  0$ are
called non-negative weights. 
We define a partial ordering on the weights by setting
$$\lambda \succeq \mu \text{ if and only if }\lambda - \mu \text{ is non-negative}.$$
Further we write $\lambda \succ \mu$ if $\lambda \succeq \mu$ and $\mu \neq \lambda$.
A weight $\lambda=\sum_{i=1}^n p_i \varepsilon_i$ with $p_1\geq p_2\geq \cdots \geq p_n$ is called dominant.

\begin{theorem}[Theorem of the highest weight]\label{thm:hwt}
	Let $W$ be an irreducible rational representation of $\GL_n(F)$. 
	\begin{enuma} \item The space 
	$$W^{U_n} = \{w \in W \mid uw = w \text{ for all } u\in U_n\}$$
	is a $1$-dimensional weight space $W_\lambda$ and all other weights $\mu$ satisfy $\mu \prec \lambda$. 
	The weight $\lambda$ is called the highest weight of $W$. 
	\item A weight $\lambda=\sum_{i=1}^n p_i \varepsilon_i$ is the highest weight of an irreducible $\GL_n(\FF)$ represenation if and only if $\lambda$ is  dominant. Moreover, $W$ is polynomial if and only if $p_n\geq 0$. 
	\item Two irreducible rational representations of $\GL_n(F)$ are isomorphic if and only if they have the same highest weight. 
	\end{enuma}

\end{theorem}

\begin{example}
	\label{ex: highest weight}
	For the irreducible $\mathrm{GL}_n(F)$ representation 
	$$ W = (\largewedge^n F^n)^{\otimes m} \otimes \largewedge^k F^n $$
	the vector 
	$$w = (e_1 \wedge \dots \wedge e_n)^{\otimes m} \otimes (e_1 \wedge \dots \wedge e_k)$$
	is a highest weight vector. The highest weight is given by
	$$ (m+1)(\varepsilon_1 + \dots + \varepsilon_k) + m (\varepsilon_{k+1} + \dots + \varepsilon_n). $$
\end{example}

Suppose we choose another basis $b'_1,\ldots, b_n'$ of $V$ to define a torus $T'\subset \GL(V)$ and characters $\epsilon'_i\colon T'\to \GL_1(\FF)$. 
Let $P\in \GL(V)$  be defined by $Pb_i' = b_i$ for $i=1,\ldots, n$. 
If $w$ is a weight vector for $T^n$ of weight $\sum_{i=1}^n p_i\epsilon_i$, then $\rho(P^{-1}) w$ is a weight vector for $T'$ of weight $\sum_{i=1}^n p_i\epsilon_i'$. Thus the partition $(p_1,\ldots, p_n)$ corresponding to the highest weight of $W$ is independent of  the choice of basis of $V$. 
\subsection{Decomposition of tensor products}

If $\rho$ is an  irreducible rational representation of $\mathrm{GL}_n(F)$, its character is the regular function 
$ \chi_\rho\colon T_n \to \FF$, 
$$ \chi_\rho(x_1,\ldots, x_n)= \tr \rho\begin{pmatrix}
	x_1 & & \\
	& \ddots & \\
	& & x_n
\end{pmatrix}.$$
A crucial property of characters is that two representations $\rho$ and $\rho'$ are isomorphic if and only if $\chi_\rho= \chi_{\rho'}$. 

 A partition is a finite sequence $\lambda=(p_1,\ldots, p_n)$ of non-negative integers. The length of $\lambda$ is the largest index $i$ such that  $p_i\neq 0$.  The number $|\lambda|= p_1+\cdots + p_n$ is called the degree of $\lambda$.

\begin{theorem} Let $L_\lambda(n)$ denote the irreducible polynomial representation with highest weight $\lambda= \sum_{i=1}^n p_i \varepsilon_i$. The character of $L_\lambda(n)$ is the Schur polynomial $s_\lambda(x_1,\ldots, x_n)$ for the partition $\lambda=(p_1,\ldots, p_n)$. 
\end{theorem}

The Schur polynomials of partitions of length $\leq n$  are a $\ZZ$ basis of the space of symmetric polynomials $\ZZ[x_1,\ldots, x_n]_{\mathrm{sym}}$. As a consequence, the product of two Schur polynomials for partitions $\lambda,\mu$ of length $\leq n$ is expressible as a unique linear combination of Schur polynomials,
\begin{equation}\label{eq:schurProduct}s_\lambda(x_1,\ldots, x_n)\cdot s_\mu(x_1,\ldots,x_n)= \sum_{\nu} N_{\nu}^{\lambda\mu} s_\nu(x_1,\ldots,x_n),\end{equation}
 where the sum is over all partitions of length $\leq n$. The non-negative integers $N_{\nu}^{\lambda\mu}$ are the Littlewood-Richardson coefficients. There is a purely combinatorial recipe for computing these numbers, see, e.g., \cite[p. 455]{FultonHarris}.

\begin{lemma}
	\label{lemma:decompGL}
	The highest weights appearing  in the isotypic  decomposition of $\Sym^pV \otimes (\Sym^q V)^*\otimes \det^q$ are of the form
$$ (p+q-i)\varepsilon_1 + q (\varepsilon_2 + \dots + \varepsilon_{n-1}) + i \varepsilon_n, \quad i=0,\dots , \min\{p,q\}.$$
	Moreover, each such highest weight appears with multiplicity one. 
\end{lemma}
\begin{proof}
Recall that the character of the tensor product of two representations $W_1, W_2$ equals the product of their characters $\chi_{W_1}\cdot \chi_{W_2}$. Since $\Sym^pV$ has highest weight $\lambda= p \varepsilon_1$, $(\Sym^q V)^*$ has highest weight $-q\varepsilon_n$,  and 
$(\Sym^qV)^*\otimes \det^q$ has highest weight $\mu= q \sum_{i=1}^{n-1}\varepsilon_i$, the lemma follows from 
\eqref{eq:schurProduct} and Pieri formula, see \cite[Appendix A]{FultonHarris}.
\end{proof}

\subsection{Representations of $\mathrm{SL}(V)$}
\label{sec:SL}

Let $V,W$  be finite-dimensional vector spaces over $F$, a field of characteristic zero. 

\begin{definition}
A representation $\rho \colon \SL(V) \to \mathrm{GL}(W)$ is called  polynomial if the matrix coefficients $\rho_{ij}$ with respect to one---and hence every---basis of $W$ are  restrictions of polynomial functions on $\End(V)$ to $\SL(V)$.
\end{definition}

Taken together, the following two propositions give us a complete description of the polynomial representations of $\SL(V)$ in terms of representations of $\GL(V)$. 

\begin{proposition}[{\cite[Proposition 5.4.2]{KraftProcesi}}]\label{prop:polySL} Every polynomial representation $\rho$ of $\SL(V)$ is the restriction of a polynomial representation $\wt \rho$ of $\GL(V)$.  Moreover, $\rho$ is irreducible if and only if $\wt \rho$ is irreducible.
\end{proposition} 

Let $\rho, \rho'$ be irreducible rational representations of $\GL(V)$ and suppose that $\rho|_{\SL(V)} = \rho'|_{\SL(V)}$. Then $\rho = \det^{ r} \otimes \rho'$ for some $r\in \ZZ$,  see \cite[Exercise 5.3]{KraftProcesi}. This fact translates into highest weights as follows. 

\begin{proposition}
Let $\rho, \rho'$ be irreducible rational representations of $\GL(V)$ with highest weights $\lambda=\sum_{i=1}^n p_i \varepsilon_i$ and $\lambda'=\sum_{i=1}^n p_i' \varepsilon_i$.  Suppose that $\rho|_{\SL(V)} = \rho'|_{\SL(V)}$. Then $p_i-p_n= p'_i - p_n'$ for $i=1,\ldots, n-1$. 
\end{proposition}

The proposition implies that there is a bijection between (the isomorphism classes of) irreducible representations of $\SL(V)$ and partitions $(p_1,\ldots, p_{n-1})$ of length at most $n-1$. We refer to the latter also as highest weights. 

Let us finally restrict the field to the case relevant to this paper, $\FF= \RR$. 
In this case, the group $\SL(V)$ is also a Lie group. Therefore it makes sense to call a representation $\rho\colon \SL(V)\to \GL(W)$ a Lie group representation if the matrix coefficients $\rho_{ij}$ are smooth functions on $\SL(V)$. Every polynomial representation is clearly smooth. The converse is also true. 

\begin{proposition}
Every finite-dimensional Lie group representation of $\SL(V)$ is polynomial.
\end{proposition}
\begin{proof}
This  follows from the representation theory of the complex Lie algebra $\sl_n(\CC)$.  Indeed, let $\rho\colon \SL_n(\RR)\to \GL(W)$ be a Lie group representation and let $d\rho\colon \sl_n(\RR)\to \End(W)$ denote its Lie algebra representation. The complexification of the latter is a Lie algebra representation $d\rho_\CC \colon \sl_n(\CC)\to \End(W_\CC)$, where $W_\CC= W\otimes_\RR \CC$. It is well known that every finite-dimensional representation of $\sl_n(\CC)$ is contained in a sum of  tensor powers $(V_\CC)^{\otimes m}\simeq (V^{\otimes m})_\CC$ of the complexification of the standard representation $V= \RR^n $, see \cite[p. 221]{FultonHarris}.   Restricting back to $\sl_n(\RR)$ and the scalar field $\RR$, it follows that $\rho$ is isomorphic to a subrepresentation of a sum of   $V^{\otimes m}\oplus V^{\otimes m}$ and hence polynomial.
\end{proof}

The following  fact will also be relevant for us. 

\begin{lemma}\label{lemma:imageSL}
	Let $\rho\colon \SL(V)\to \GL(W)$ be a polynomial representation. Then the image of $\rho$ is contained in $\SL(W)$. 
\end{lemma}
\begin{proof}
	Let $[X,Y]=XY-YX$ denote the Lie bracket in $\End(V)$ and its Lie subalgebras. Since $\sl(V)=[\sl(V),\sl(V)]$, $d\rho$ respects the Lie bracket,  and 
	$\tr[X,Y]=\tr(XY-YX)=0$, it follows that that the image of $d\rho$ is contained in $\sl(W)$. Using $\rho(e^X)= e^{d\rho(X)}$, the claim follows.  
\end{proof}

\section{Reduction to homogeneous  even  valuations}

For the rest of the paper let $V$ be a finite-dimensional real vector space of dimension at least $2$. Let $\rho\colon \SL(V)\to \GL(W)$ be a polynomial representation of $\SL(V)$. To distinguish below more precisely between different types of invariance, a  map $\calK(V)\to \calK(W)$  is called $\SL(V)$-invariant if $\rho(T)\Phi(T^{-1} K)=  \Phi K$ holds for all $K$ and $T\in \SL(V)$.  

Recall that  a Minkowski valuation $\Phi\colon \calK(V)\to \calK(W)$ is said to be non-trivial, if there exists  $K$  such that $\Phi(K)\neq \{0\}$. Otherwise we call $\Phi$ trivial.  
The goal of this section is to show that if there exists at all a non-trivial $\SL(V)$-invariant Minkowski valuation $\Phi\in \MVal(V,W)$, then there exists also a non-trivial Minkowski valuation with particularly desirable properties.

\begin{proposition}\label{prop:nicer}
		Let $W$ be a non-trivial irreducible representation of $\SL(V)$. If there exists at least one  non-trivial  $\SL(V)$-invariant Minkowski valuation in $\MVal(V,W)$, then there exists also such a valuation that is in addition homogeneous, even and takes values in the  subset of origin-symmetric convex bodies. 
	
\end{proposition}

We will need the following lemma.

\begin{lemma}\label{lemma:vecValued}
	Let $W$ be a non-trivial irreducible representation of $\SL(V)$. If  $\phi\colon \calK(V)\to W$ is a translation-invariant, continuous and $\SL(V)$-equivariant valuation, then $\phi=0$. 
\end{lemma}
\begin{proof}
Consider the decomposition $\phi=\phi_0+ \cdots + \phi_n$ into homogeneous components. Since $\phi$ is $\SL(V)$-equivariant, every homogeneous component has this property as well. Without loss of generality we may therefore  assume that $\phi$ is homogeneous. Let us fix a density $\vol$ on $V$. 

If $\phi$ is homogenous of degree $n$, then there exists by Hadwiger's theorem a vector $w\in W$ such that $$\phi(K) =\vol(K) w, \quad K\in \calK(V).$$ 
If $\phi \neq 0$, then the $\SL(V)$-equivariance of $\phi$ and the irreducibility of $W$ imply that $W$ is spanned by $w$. This contradicts the assumption that $W$ is non-trivial. 

A similar, but even simpler argument shows that  $\phi=0$ in the case that $\phi$ is homogeneous of degree $0$. 

We  assume from now on that $\phi$ is homogeneous of degree $0<k<n$. Splitting $\phi= \phi^+ + \phi^-$ into an even and an odd part we may further assume that $\phi$ is either even or odd.  Fix some $T\in \GL(V)$ with $\det T= -1$ and consider the finite-dimensional subspace $U$ of $\Val_k(V)$ spanned by the valuations $ \langle \phi(K),\xi\rangle$ and $ \langle \phi(T(K)),\xi\rangle$  for $\xi\in W^*$.  Observe that $U$ is a $\GL(V)$-invariant subspace.
Since $\Val_k(V)$ is infinite-dimensional for $0<k<n$, Alesker's irreducibility theorem implies $U=\{0\}$. 
\end{proof}

\begin{proof}[Proof of Proposition~\ref{prop:nicer}] Let $\Psi\in \MVal(V,W)$. For each $\xi\in W^*$ consider the valuation $\psi_\xi\colon \calK(V)\to \RR$ defined by 
	$ \psi_\xi(K) = h_{\Psi K} (\xi)$. Let 
	$$ \psi_\xi = \psi_{\xi,0} + \cdots + \psi_{\xi,n}$$
	be the decomposition into homogeneous components. Let $i_0$ be the smallest integer $i$  with the property that $\psi_{\xi,i}$ is for some $\xi$ not identically zero. 
	Obviously,
	$$ \psi_{\xi,i_0}(K)  =  \lim_{t\to\infty}  t^{-i_0}h_{\Psi(tK)}(\xi)$$
	holds for every convex body $K$ and all $\xi\in W^*$. As a pointwise limit of support functions, the function $\xi \mapsto \psi_{\xi,i_0}(K)$ is the support function of a unique convex body in $W$ that we denote by $\Phi(K)$. It is clear from this definition that  $\Phi\colon \calK(V)\to \calK(W)$ is non-trivial,  translation-invariant, continuous, $\SL(V)$-invariant, and homogeneous of degree $i_0$. 
	The Minkowski valuation $ \wt \Phi \colon \calK(V)\to \calK(W)$
	$$ \wt \Phi(K)=  \Phi(K)+  \Phi(-K) + (- ( \Phi(K)+ \Phi(-K)))$$
	is not only translation-invariant, continuous, $\SL(V)$-invariant, and homogeneous, but also even and satisfies $-\wt \Phi(K)= \wt \Phi(K)$. 
	If $ \wt \Phi$ was  trivial, then $ \Phi $ would be vector-valued, i.e.\ $\Phi K = \{ \phi (K)\}$ for some valuation $\phi\colon \calK(V)\to W$. Since $\phi$ satisfies the assumption of Lemma~\ref{lemma:vecValued}, this  implies that already $\Phi$ was trivial, a contradiction.

\end{proof}

\section{Proof of Theorem~\ref{thm:mainA}}

As before let $V$ denote a real vector space of dimension $n\geq 2$ and let $\rho\colon \SL(V)\to \GL(W)$ be a  non-trivial irreducible polynomial $\mathrm{SL}(V)$-representation. As we saw in Proposition~\ref{prop:polySL}, the representation $\rho$ is the restriction of an irreducible polynomial $\mathrm{GL}(V)$-representation. We have also seen  (Proposition~\ref{prop:TensorPow}) that $W$ is a subrepresentation of some tensor power of $V,$ say $W \subset V^{\otimes d}$.

 In this section, we assume $\Phi  \in \MVal_k^+(V,W)$ to be non-trivial, $\mathrm{SL}(V)$ equivariant and to  satisfy $\Phi = -\Phi$. We have already seen in Proposition~\ref{prop:nicer} that if there exists a non-trivial and $\SL(V)$-invariant $\Phi\in\MVal(V,W)$, then there exists also a Minkowski  valuation with these additional properties.   
  The goal of this section is to prove that these assumptions already imply  $W \simeq \largewedge^k V.$ We achieve this by determining the Klain section of $\Phi$ and we conclude using Theorem~\ref{thm:imageKlain}. 
 
 Let $\GL^+(V)\subset \GL(V)$ denote the subgroup of automorphisms  with positive determinant.

\begin{lemma}
	\label{lemma: weak GL+ equivariant}
	For all $T \in \mathrm{GL}^+(V)$ we have $$\Phi(T K) =\det(T)^{-m} \rho(T) \Phi(K),$$ 
	where
	$$m:= \frac{d-k}{n}.$$ 
\end{lemma}

\begin{proof}
	Let $T \in  \mathrm{GL}^+(V).$ Then $(\det T)^{-\frac{1}{n}} T \in \mathrm{SL}(V)$ and by homogeneity 
	$$ \Phi(T K) = (\det T)^{\frac{k}{n}} \cdot \Phi\left((\det T)^{-\frac{1}{n}} T K\right).$$ By $\mathrm{SL}(V)$-invariance  the right-hand side is equal to 
	$$(\det T)^{\frac{k}{n}} \cdot \rho\left((\det T)^{-\frac{1}{n}} T \right) \Phi( K)  = \det(T)^{-m} \rho(T) \Phi(K).$$ For the last equality, we used the assumption $W\subset V^{\otimes d}.$ 
\end{proof}

By $\SL(V)$-invariance and Proposition~\ref{prop:KlainMink}, our Minkowski valuation $\Phi$ is determined by the value of its Klain section on a single $k$-dimensional linear subspace $E$. 
It is at this point  convenient to fix a euclidean structure on $V\simeq \RR^n$. This choice normalizes the Lebesgue measure on $V$ and all of its subspaces.  
Fix  $E=\RR^k\times \{0\}\subset \RR^n$ and let $L_0$ be the convex body in $W$ with the property that 
$$ \Klain_\Phi(E)(K)=  \vol_k(K) L_0 
$$ for all $K\in \calK(E)$.

\begin{lemma}
	\label{lemma:invariant formula}
	Let  $T \in \GL_n^+(\RR)$ be  such that $TE=E.$ Then we have 
	$$ |\det (T|_{E})| L_0 = \det(T)^{-m} \rho(T) L_0. $$
\end{lemma}

\begin{proof}
	Let $K \subset E$ be a convex body with $\vol_k(K)=1.$ By the definition of $L_0$ we have 
	$$ \Phi(T K) = \vol_k(T K) \cdot L_0 = |\det(T|_{E })|  \cdot L_0. $$
	On the other hand, using Lemma  \ref{lemma: weak GL+ equivariant} we obtain
	$$ \Phi(T K) = \det(T)^{-m}\rho(T) \Phi(K) = \det(T)^{-m}  \rho(T) L_0  $$
	and hence the claim.
\end{proof}

Let  $e_1, \dots , e_n$ denote the standard basis of $\RR^n.$ The previous lemma implies that $L_0$ is invariant under  the subgroup $U_n\subset \GL_n(\RR)$ of unipotent upper triangular matrices. The next goal is to show that $L_0$ is contained in the highest weight space with respect to the standard basis.

\begin{lemma}
	\label{lemma: trivial action of Un on the Klain body}
	If $w \in L_0$, then $U_n$ acts trivially on $w$.
\end{lemma}

\begin{proof}
	Since $W$ is a $\GL_n(\RR)$-representation we have a decomposition into weight spaces
	$$ W = \bigoplus_{\lambda \in \Lambda} W_\lambda,$$
	where $\Lambda$ denotes the set of weights of $W.$ Now let 
	$$w = \sum\limits_{\lambda \in \Lambda} w_\lambda\in L_0$$
	with $w_\lambda \in W_\lambda.$ The group $U_n$ is generated by the elements $u_{ij}(s) = \id + s E_{ij}$ for $j>i.$ Hence it suffices to show that $\rho(u_{ij}(s))$ acts trivially on $w.$ By Lemma~\ref{lemma: Un action on weight vectors} there are $w_{\lambda,l} \in W_{\lambda + l(\varepsilon_i - \varepsilon_j)}$ for $l\geq 0$ such that $w_{\lambda,0} = w_\lambda$ and 
	$$ \rho(u_{ij}(s))w_\lambda = \sum\limits_{l\geq 0} s^lw_{\lambda,l}. $$
	Note that this is a finite sum since $W$ has finite dimension. Hence 
	$$ \rho(u_{ij}(s)) w = \sum\limits_{\lambda \in \Lambda} \sum\limits_{l\geq 0}  s^lw_{\lambda,l}  = \sum\limits_{\lambda \in \Lambda} w_{\lambda,0} + s \sum\limits_{\lambda \in \Lambda} w_{\lambda,1 } + \cdots $$
	Since $L _0$ is invariant under $U_n$, we have $\rho(u_{ij}(s)) w \in L_0$ for all $s \in \RR.$ But $L_0$ is compact and therefore letting $s \to \infty$ we see that $\sum\limits_{\lambda \in \Lambda} w_{\lambda, i}$ has to be zero for $i >0$. Hence
	$$\rho(u_{ij}(s)) w = \sum\limits_{\lambda \in \Lambda} w_{\lambda,0} = \sum\limits_{\lambda \in \Lambda} w_\lambda = w. $$
\end{proof}

\begin{corollary}
	\label{corollary: klain body is a line segment}
	There is a highest weight vector $w \in W$ such that $L_0=[-w,w].$ 
\end{corollary}
\begin{proof}
	We have just seen that  $L_0 \subset W^{U_n}$.  But  $W^{U_n}$ equals by Theorem~\ref{thm:hwt} the highest weight space, which is always of dimension $1.$ Since $\Phi = -\Phi$ we have $L_0 = [-w,w]$ for a highest weight vector $w.$
\end{proof}

It remains to determine the highest weight of the representation $W.$ 

\begin{proposition}
	\label{proposition: W = lambdakV}
	The highest weight of $W$ is 
	$$\lambda = (m+1)(\varepsilon_1 + \dots + \varepsilon_k) +m(\varepsilon_{k+1} + \dots + \varepsilon_n).$$ Consequently, $W\simeq \largewedge^k V$ as $\mathrm{SL}_n(\RR)$-representations.
\end{proposition}

\begin{proof}
	We compute the action of the $n$-dimensional torus $T^n$ on the highest weight vector $w$ to find the highest weight of $W.$ Let 
	$$ T = \begin{pmatrix}
		t_1 & & \\
		& \ddots & \\
		& & t_n
	\end{pmatrix} \in T^n. $$
 The highest weight $p_1\varepsilon_1+ \cdots + p_n\varepsilon_n$ of $W$  satisfies
 $$ \rho(T) w = t_1^{p_1} \cdots t_n^{p_n}  w, \quad T\in T_n. $$
At the same time, by Lemma \ref{lemma:invariant formula}, we have 
	$$ t_1 \cdots t_k [-w,w] = (t_1 \cdots  t_n)^{-m} \rho(T)[-w,w],$$
	provided all $t_i$ are positive. Hence 
$$		\rho(T)w =  (t_1 \cdots t_k)^{m+1}\cdot (t_{k+1} \cdots t_n)^mw$$
for all $T$.
Thus 	Example \ref{ex: highest weight}  and the theorem of the highest weight (Theorem~\ref{thm:hwt}) imply
	$$ W \simeq (\largewedge^n \RR^n)^{\otimes m} \otimes \largewedge^k\RR^n $$
	as $\GL_n(\RR)$-representations. Since the factor $(\largewedge^n \RR^n)^{\otimes m}$ is just multiplication with a power of the determinant, $W\simeq \largewedge^k\RR^n$ as  $\mathrm{SL}_n(\RR)$-representations. 
\end{proof}

From now on we assume that $W=\largewedge^k \RR^n$  for some $k\in \{1,\ldots, n-1\}$.
Given a  $k$-vector $v\in \largewedge^k \RR^n$ we define a function $f_v$ on  $\Grass_k(\RR^n)$  as follows:
$$f_v(U)= |\langle u_1\wedge \cdots \wedge u_k, v\rangle |,\quad U\in \Grass_k(\RR^n),$$ 
where $u_1,\ldots, u_k$ is an orthonormal basis of $U$ and $\langle \,\cdot\,,\,\cdot\,\rangle$ denotes the usual inner product on $\largewedge^k \RR^n$. Note that this is well-defined, i.e., is independent of the choice of orthonormal basis of $U$.
For  $v\in \largewedge^k \RR^n$, let $h_v$ denote the Klain function of $K\mapsto h_{\Phi K}(v)$. 

\begin{lemma} \label{lemma:h=f}There exists a positive constant $c$ such that 
	$ h_v= c f_v$ for all $k$-vectors $v\in \largewedge^k \RR^n$.
\end{lemma}
\begin{proof}
Recall from Example \ref{ex: highest weight} that the highest weight space of $\largewedge^k \RR^n$ is spanned by $e_1\wedge \cdots \wedge e_k$. Let $U\subset \RR^n$ be a $k$-dimensional linear subspace and let $E \subset \RR^n$ be the subspace spanned by $e_1, \dots , e_k$. Choose $T\in \SL_n(\RR)$ such that $U= T E $. Let $K\subset E$ be  a convex body with $\vol_k(K)=1$. By the  definition of $h_v$ we have 
\begin{align*}  \vol_k(TK) h_v(U) & = h_{\Phi(TK)}(v) =  h_{\Phi K} (T^{*} v) \\
	& = h_{T^{*} v} (E) = c |\langle e_1\wedge \cdots \wedge  e_k , T^{*} v\rangle | \\
	&= c |\langle T e_1\wedge \cdots\wedge T e_k , v\rangle | 
	\end{align*}
Since $|T  e_1\wedge \cdots \wedge  Te_k| = \vol_k(TK)$, the claim follows.
\end{proof}

Let  $C(\Grass_k(\RR^n))$ denote the space of continuous functions on the Grassmannian equipped with the topology of uniform convergence. 

\begin{proposition} \label{prop:dense}Let $k\in \{1,\ldots, n-1\}$. 
	The functions $f_v$, $v\in \largewedge^k\RR^n$,    span a  dense  subspace of $C(\Grass_k(\RR^n))$.
\end{proposition}

\begin{proof}
	Let us first point out that the statement of the theorem is well-known in the cases $k=1$ and $k=n-1$. In fact,  these cases are an immediate consequence of the  injectivity of the spherical cosine transform   on even functions (see, e.g., \cite{Groemer:Harmonics}).
	
	We assume therefore from now on that $1<k<n-1$. Let $\Grass_k^+(\RR^n)$ denote the Grassmannian of oriented $k$-dimensional linear subspaces. Let  $\iota \colon \Grass_k^+(\RR^n) \to \largewedge^k \RR^n$ denote the Pl\"ucker embedding, 
	$$ \iota(U) = u_1\wedge\cdots \wedge u_k, \quad U\in \Grass_k^+(\RR^n),$$
	where $u_1,\ldots, u_k$ is a positively oriented orthonormal basis of $U$. 
	
	The functions on   the usual Grassmannian are  in one-one correspondence with the even functions on the oriented Grassmannian. In fact, the latter are pullbacks of the former under the canonical projection $\pi\colon \Grass_k^+(\RR^n)\to \Grass_k(\RR^n)$. 
	Since the image of the Pl\"ucker embedding is a closed subset, every continuous function on $\Grass_k^+(\RR^n)$ is the restriction of some continuous function on $S(\largewedge^k \RR^n)$, the unit sphere in $\largewedge^k\RR^n$. The same conclusion holds, if we consider only even functions.
	
	Let $g$ be an even continuous function on $\Grass_k^+(\RR^n)$. Choose an even function $\wt g$ on the sphere $S(\largewedge^k \RR^n)$ such that $\iota^* \wt g = g$. It follows from the special case $k=1$ of Proposition~\ref{prop:dense} that 
	$\wt g$ can be approximated uniformly  by linear combinations of functions of the form 
	$$\wt f_v(u) = |\langle u, v\rangle|, \quad u\in S(\largewedge^k \RR^n),$$ 
	where $v\in \largewedge^k\RR^n$. Since $\iota^* \wt f_v =\pi^* f_v$, the claim follows.
	
\end{proof}

\begin{proof}[Proof of Theorem \ref{thm:mainA}] Assume that  the representation $W$ is irreducible and  non-trivial.
	If there exists a non-trivial and $\mathrm{SL}(V)$-invariant Minkowski valuation in $\MVal(V,W)$, then by Proposition \ref{prop:nicer}, there exists also  for some $k \in \{0,\dots , n\}$  a  non-trivial and $\SL(V)$-invariant Minkowski valuation  $\Phi \in \MVal_k^+(V,W)$ satisfying  $\Phi = -\Phi.$ 
	
	Under these assumptions, Proposition \ref{proposition: W = lambdakV} implies $W \simeq  \largewedge^k V$ for some  $k\in \{ 1,\ldots, n-1\}$.  Thus, in view of Lemma~\ref{lemma:h=f} and Proposition~\ref{prop:dense}, the existence of  $\Phi$ implies that the image of the Klain embedding $\Val_k^+(\RR^n)\to C(\Grass_k(\RR^n))$  is dense. By Theorem~\ref{thm:imageKlain} this can happen  only for  $k=1$ or $k=n-1$. 
\end{proof}

\section{The reducible case}

\label{sec:reducible}	
We can use Theorem~\ref{thm:mainA} to obtain information about Minkowski valuations with values in $\calK(W)$ where $W$ is reducible. 
We will first show that if $\Phi$ is translation-invariant  and takes values in $\calK(W)$ where $W=\RR^p \oplus V^q \oplus (V^*)^r$, then $\Phi$ splits nicely into a direct sum of Minkowski valuations. Throughout this section $V$ is again a real vector space of dimension $n\geq 2$ and $\vol$ is a fixed positive density on $V$.

\begin{proposition}\label{prop:split}
	Let $W=  \RR^p \oplus V^q \oplus (V^*)^r $ and let $\Phi\colon \calK(V)\to \calK(W)$  be a translation-invariant, $\SL(V)$-invariant, continuous Minkowski valuation. Then there exist convex bodies $L_0,L_n\subset \RR^p$ and a translation-invariant, $\SL(V)$-invariant, continuous Minkowski valuation $\Psi\colon \calK(V)\to \calK(V^q\oplus (V^*)^r)$ such that
	$$ \Phi(K)= L_0 + \Psi(K) + \vol(K) L_n$$
\end{proposition}
\begin{proof}
	
	There exist Minkowski valuations $\Phi_0,\Phi_n\colon\calK(V)\to \calK(W)$ of degrees $0$ and $n$  and for each degree $i\in \{1,\ldots, n-1\}$ a family of  valuations  $\psi_i(\,\cdot\,, \xi) \in \Val_i(V)$,  $\xi\in W^*$, 
	such that 
	\begin{equation}\label{eq:hom_decomp} h( \Phi K, \xi)=  h(\Phi_0K, \xi) + \sum_{i=1}^{n-1} \psi_i(K,\xi)+ h(\Phi_nK,\xi)\end{equation}
	for all $K$.

	Let $\pi$ denote a projection onto one of the summands $V$ or $V^*$ of $W$. Then $\pi \circ \Phi_i$, $i=0,n$, are trivial by Ludwig's theorem (Theorem~\ref{thm:Ludwig}). Since  there exist convex bodies $L_0$ and $L_n$ such that $\Phi_0(K)= L_0$ and $\Phi_n(K)= \vol(K) L_n$ this implies that $L_0,L_n\subset \RR^p$.	
	
	Let us from now on view $W^*$ as $\RR^p \oplus (V^q\oplus (V^*)^r)^*$ and let us write $\xi= (t,\lambda)\in W^*$ accordingly. 
	Fix $t\in \RR^p$.   
	Then $K\mapsto h(\Phi K, (t,0))$ is an $\SL(V)$-invariant, translation-invariant valuation, hence
	there  exist constants such that $h(\Phi K, (t,0))= c_0 + c_n\vol(K)$. It follows that 
	\begin{equation} \label{eq:proj_pi_1} \psi_i(K,(t,0))=0 \quad \text{for all } t\in \RR^p.\end{equation}

	Put $\mu_i(K,\xi)= \psi_i(K,(t,\lambda))- \psi_i(K,(0,\lambda))$, $\mu(K,\xi)= \sum_{i=1}^{n-1} \mu_i(K,\xi)$ and $\psi(K,\xi)= \sum_{i=1}^{n-1} \psi_i(K,\xi)$. Our goal is to show that $\mu(K,\xi)=0$ for all $K$ and $\xi$. If we can accomplish this, the proposition is proved, because then $\psi(K,\xi)=h(\pi_2 \Phi K , \xi)$, where $\pi_2$ denotes the projection $W\to V^q\oplus (V^*)^r$. 
	
	If $K$ is a   convex body contained in  a hyperplane, then $\xi \mapsto \psi_{n-1}(K,\xi)$ is the  support function of a  convex body. By equation \eqref{eq:proj_pi_1} this convex body must be contained in $V^q\oplus (V^*)^r$. It follows that $\mu_{n-1}(K,\xi)=0$ for all $\xi$. Klain's theorem (Theorem~\ref{thm:Klain}) implies that $\mu_{n-1}(\,\cdot\, ,\xi)$ is an odd valuation.
	
	Let us write 
	$$ h(\Phi K , \xi ) =  \psi(K,\xi)+ \phi(K,\xi)$$
	Recall that $\phi(K,(t,\lambda))= \phi(K,(t,0))$ and $\psi(K,(t,0))=0$. Therefore   
	$$   h(\Phi K , (t,\lambda)) \leq  h(\Phi K , (t,0))  + h(\Phi K , (0,\lambda))$$
	clearly  implies 
	$$ \psi(K,(t,\lambda))+ \phi(K,(t,0))  \leq   \phi(K,(t,0))  + \psi( K , (0,\lambda)).$$
	and thus 
	\begin{equation}\label{eq:psi_sign} \mu(K,\xi) \leq 0.\end{equation}
	Replacing $K$ by $sK$ and letting $s\to \infty$, this equation shows that  $\mu_{n-1}(K,\xi)\leq 0$ for all $K$ and $\xi$. But $\mu_{n-1}(K,\xi)$ is  an odd valuation. Therefore $\mu_{n-1}(K,\xi)=0$ for all $K$ and $\xi$. 
	
	Suppose that $\mu_i(\,\cdot\,,\xi)=0$ for all $\xi \in W^*$ and $i>j$. Let $K$ be a  convex body contained in a linear subspace of dimension $j$. Then $\psi_i(K,\xi)=0$ for all $i>j$ and thus $\xi \mapsto \psi_j(K,\xi)$ is the support function of a convex body. Again by \eqref{eq:proj_pi_1} this convex body is contained in $V^q\oplus (V^*)^r$ and  so $\mu_{j}(K,\xi)=0$ for all $\xi$. By Klain's theorem, $ \mu_j(\,\cdot\,,\xi)$ is an odd valuation. Since  $\mu_i(\,\cdot\,,\xi)=0$ for all $i>j$,  equation \eqref{eq:psi_sign} implies as before that $\mu_j(K,\xi)\leq 0$ for all $K$ and $\xi$. Therefore $\mu_i(K,\xi)=0$. We conclude that $\mu(K,\xi)=0$ for all $K$ and $\xi$. 	
\end{proof}

\begin{lemma} \label{lemma:split} Let 
	$\Phi\colon \calK(V)\to \calK(V^q\oplus (V^*)^r)$ be a translation-invariant, $\SL(V)$-invariant, continuous Minkowski valuation. There exist  translation-invariant, $\SL(V)$-invariant, continuous Minkowski valuations $\Phi_1\colon \calK(V)\to \calK(V^q)$ and $\Phi_{n-1}\colon \calK(V)\to \calK((V^*)^r)$ such that 
	$$ \Phi(K)=\Phi_1(K)+ \Phi_{n-1}(K)$$
	for all convex bodies $K\in \calK(V)$. Moreover, $\Phi_1$ and $\Phi_{n-1}$ are homogeneous of degree $1$ and $n-1$.	
\end{lemma}
\begin{proof}
Put $W=V^q \oplus (V^*)^r$.  We have  the homogeneous decomposition \eqref{eq:hom_decomp}. Composing with projections onto the summands $V$ or $V^*$ we conclude that $\Phi_0$ and $\Phi_n$ are trivial.
	 Hence for $i=1,n-1$ there exist translation-invariant, $\SL(V)$-invariant,  continuous Minkowski valuations $\Phi_i\colon \calK(V)\to \calK(W)$, such that $\psi_i(K,\xi)= h(\Phi_iK, \xi)$. Composing with projections to the summands $V$ or $V^*$,  Ludwig's theorem (Theorem~\ref{thm:Ludwig}) implies $\Phi_1(K)\subset V^q$ and $\Phi_{n-1}(K)\subset (V^*)^r$ for all $K$. 
	 
	 Put  $\mu(K,\xi)= \sum_{i=2}^{n-2} \psi_i(K,\xi)$. The lemma is proved if we can show that $\mu(K,\xi)=0$ for all $K$ and $\xi$. 
	 Let us write $\xi=\xi'+ \xi''\in W^*= (V^q)^* \oplus V^r$.  Then 
	 $$ h(\Phi K, \xi') = h(\Phi_1K, \xi') +\mu(K,\xi')$$ 
	 and it follows that $\xi'\mapsto \psi_{n-2}(K,\xi')$ is the support function of a convex body in $V^q$. At the same time, again by Ludwig's theorem, the projection of this convex body onto each summand $V$ vanishes. Thus $\psi_{n-2}(K,\xi')=0$. Proceeding inductively, we conclude that $\mu(K,\xi')=0$ for all $\xi'$ and convex bodies $K$. An analogous argument shows that $\mu(K,\xi'')=0$. 
	 
    Writing $\xi=\xi'+ \xi''\in W^*= (V^q)^* \oplus V^r$, we clearly have
	\begin{align*} \psi_1(K, \xi')+ \mu(K,\xi)+ \psi_{n-1}(K,\xi'') & =  h(\Phi K,\xi)\\
		& \leq h(\Phi K,\xi') + h(\Phi K,\xi'')\\
		& = \psi_1(K, \xi')+ \mu(K,\xi')+ \mu(K,\xi'')+ \psi_{n-1}(K,\xi'')
	\end{align*}
	and hence $ \mu(K,\xi)\leq \mu(K,\xi')+ \mu(K,\xi'')$.
	Since $\mu(K,\xi')=\mu(K,\xi'')=0$, we obtain 
	\begin{equation}\label{eq:muleq0}\mu(K,\xi)\leq 0\end{equation}
	for all $K$ and $\xi\in W^*$. Moreover,  
	$$ h(\Phi K,\xi')\leq h(\Phi K, \xi'+ \xi'')+ h(\Phi K,-\xi'')$$ 
	implies 
	\begin{equation}\label{eq:mugeq0} 0\leq  \mu(K,\xi) +    \psi_{n-1}(K,\xi'') + \psi_{n-1}(K,-\xi'').\end{equation}
	Replacing $K$ by $tK$ for $t>0$ and letting $t\to 0$, the inequalities \eqref{eq:muleq0} and \eqref{eq:mugeq0} imply $\psi_2(K,\xi)=0$. Inductively we obtain $\mu(K,\xi)=0$, as desired.
\end{proof}

\begin{proof}[Proof of Corollary~\ref{cor:reducible}]
By Propositions~\ref{prop:comp_red} and \ref{prop:polySL}, every representation $W$ of $\SL(V)$ decomposes into irreducible representations, $W= W_1\oplus \cdots \oplus W_m$. Let $\pi_i$ denote the projection onto the $i$-th summand. By Theorem~\ref{thm:mainA}, the composition  $\pi_i\circ \Phi$ is trivial unless $W_i$ is isomorphic to $\RR$, $V$, or $V^*$. Let $U\subset W$ denote the subspace spanned by summands $W_i$ with this property. Then $U= \RR^p\oplus V^q \oplus (V^*)^r$ for certain numbers $p,q,r$ and $\Phi(K)\subset U$ for all $K$.  Applying now Proposition~\ref{prop:split} and Lemma~\ref{lemma:split} yields the desired splitting.
\end{proof}

Let us close this section with a discussion of examples of translation-invariant, $\SL(V)$-invariant, continuous Minkowski valuations taking values in $\calK(V^p)$ or $\calK((V^*)^p)$. 

Let us first recall that the dual space $\Hom(V,W)^*$ of linear maps from $V$ to $W$ is for arbitrary real finite-dimensional vector spaces $V,W$ naturally isomorphic to $\Hom(W,V)$. Indeed, the trace,
\begin{align}
	\label{eq:trace_pairing}
	\begin{split}
	\Hom(V,W)\times \Hom(W,V)&\to \RR,\\
	  (S,T)&\mapsto \tr(T\circ S),
	 \end{split}
\end{align} 
induces a non-degenerate pairing.
Moreover, let us point out that $V^p$ and $(V^*)^p$ are naturally isomorphic to $\Hom(\RR^p,V)$ and $\Hom(V,\RR^p)$, respectively. 

In the following we use the standard inner product on $\RR^p$ to identify $\RR^p$ with its dual space.
Let $Q$ be a convex body in $\RR^p$. One immediately verifies that 
\begin{equation}\label{eq:Qprojbody} h_{\Pi_Q K}(\xi)= V(K,\ldots, K, h_Q\circ \xi^*), \quad \xi\in \Hom(\RR^p,V),\end{equation}
where $V(K_1,\ldots, K_n)$ denotes the mixed volume for a fixed positive density $\vol$ on $V$ and $\xi^*\in \Hom(V^*,\RR^p)$ is the dual map to $\xi$,  is the support function of a convex body $\Pi_QK$ in $(V^*)^p \simeq \Hom(V,\RR^p)$. Moreover,  $\Pi_Q$ is a translation-invariant, $\SL(V)$-invariant, continuous Minkowski valuation. The operator $\Pi_Q K$ was introduced in \cite{Haddad_etal:LpIsoperimetric} and called the $Q$-projection body of $K$. The special case where $Q$ is the unit cube was previously studied in \cite{Haddad_etal:GenProj}. A more geometric reformulation of \eqref{eq:Qprojbody} is
\begin{equation}\label{eq:Qproj2}
	h_{\Pi_QK}(\xi)= V(K,\ldots, K, \xi(Q)), \quad \xi\in \Hom(\RR^p,V).
\end{equation} 

We can dualize \eqref{eq:Qprojbody} to obtain a family of $1$-homogeneous Minkowski valuations as follows. 
\begin{proposition}
Let $\mu$ be a Borel measure on the unit sphere $S^{p-1}\subset \RR^p$ such that 
\begin{equation}\label{eq:mu_centered} \int_{S^{p-1}} u\, d\mu(u)=0.\end{equation}
 Define for $K\in \calK(V)$
$$ h_{M_\mu K}(\xi)= \int_{S^{p-1}} h_K\circ \xi^*(u) \, d\mu(u), \quad \xi\in \Hom(V,\RR^p),$$
where $\xi^*\in \Hom(\RR^p,V^*)$ is the dual map to $\xi$.  
Then the following properties hold:
\begin{enuma}
\item $h_{M_\mu K}$ is the support function of a convex body $M_\mu K\subset V^p$. 
\item $M_\mu(K+L)= M_\mu K + M_\mu L$ for all convex bodies $K,L\subset V$. 
\item $M_\mu\colon\calK(V)\to \calK(V^p)$ is a translation-invariant, $\SL(V)$-invariant, continuous Minkowski valuation.  
\end{enuma}
\end{proposition}

Before we prove this proposition, let us give a more geometric description of $M_\mu K$. Let $p\geq 2$ and suppose that there exists  a convex body $Q\subset \RR^p$ with surface area measure equal  to $p\cdot  d \mu$. By Minkowski's existence theorem \cite[Theorem 8.2.2]{Schneider:BM} such a body exists if and only if  $\mu$ is not concentrated on a great subsphere. Then we may write 
$$  h_{M_\mu K}(\xi)= V(\xi(K),Q,\ldots, Q), \quad \xi\in \Hom(V,\RR^p).$$
This expression should be compared with \eqref{eq:Qproj2}. We  call $M_QK=M_\mu K$ the \emph{$Q$-mean width body} of $K$. 
\begin{proof}
To see (a) just observe that for $\xi_1,\xi_2\in \Hom(V,\RR^p)$ one has
$$ h_K\circ(\xi_1+\xi_2)^*(u)= h_K(\xi_1^*(u) + \xi_2^*(u))\leq h_K\circ \xi_1^*(u) + h_K\circ \xi_2^*(u)$$
and use that $\mu$ is non-negative by assumption.

(b) is straightforward. A consequence of (b) and \eqref{eq:mu_centered} is that $M_\mu$ is translation-invariant. 
The valuation property follows from the well-known fact that $h_{K\cup L}+ h_{K\cap L} = h_K + h_L$ whenever $K\cup L$ is convex. Since continuity is obvious, it remains to show $\SL(V)$-invariance. For  $T\in \SL(V)$ one has
\begin{align*} h_{M_\mu (TK)}(\xi)& = \int_{S^{p-1}} h_{K}\circ  T^* \circ \xi^* (u) \, d\mu(u)=\int_{S^{p-1}} h_{K}\circ  (\xi \circ T)^* (u) \, d\mu(u)\\
	&= h_{M_\mu K}(\xi\circ T) =\sup_{S\in M_\mu K} \tr (\xi\circ T \circ S) \\
	&= h_{TM_\mu K}(\xi),
\end{align*}
where in the penultimate equality we have used \eqref{eq:trace_pairing}. This finishes the proof.
\end{proof}

\section{New  Minkowski valuations for irreducible representations}

Let $V$ be a real vector space of dimension $n  $ at least $2$.

\subsection{Symmetric tensors}\label{sec:symPowers}

Before we present the main construction of this section, let us be more explicit   about certain properties of the symmetric tensor powers. 

For us the symmetric algebra $\Sym V$  of a vector space $V$ is the quotient of the tensor algebra $ TV = \bigoplus_{p= 0}^\infty V^{\otimes p} $ by the two-sided ideal generated by $u\otimes v-v\otimes u$. The grading of the tensor algebra is inherited by $\Sym V = \bigoplus_{p=0}^\infty \Sym^p V$. Given $x\in V$ we write  $x^p =  x\otimes \cdots \otimes x\in V^{\otimes p}$ and by $x^p\in \Sym^pV$ we mean the image of the latter expression under the natural projection $TV\to \Sym V$. Recall that the symmetric tensor powers have the universal  property that they linearize symmetric multilinear maps $V^p \to W$. 

Consider the $\GL(V)$ equivariant pairing $\Sym^p V\times \Sym^p V^*\to \RR$ defined by 
\begin{equation}\label{eq:pairing}((v_1,\ldots, v_p),( \xi_1,\ldots, \xi_p))\mapsto  \sum_{\sigma\in S_n}  \xi_{\sigma(1)} (v_1)\cdots
	\xi_{\sigma(p)} (v_p).\end{equation}
Since this pairing is non-degenerate, we get a $\GL(V)$ equivariant isomorphism $(\Sym^p V)^*\simeq \Sym^p V^*$ that we will often use without explicit  mention.

Let us fix a euclidean inner product on $V$ so that we can identify $V\simeq \RR^n$. Combining  $V^*\simeq V$ with the pairing \eqref{eq:pairing} defines the euclidean inner product
$$ \Sym^p V\times \Sym^p V\simeq \Sym^p V\times \Sym^p V^* \to \RR$$
that we will use in the following. 
Note that  this inner product satisfies
\begin{equation}\label{eq:innerprodTensor} \langle u^p,v^p \rangle = p! \langle u,v\rangle^p 
\end{equation}
for all $u,v\in V$.

\subsection{The main construction}

The goal of this section is to prove Theorem~\ref{thm:newMink}. 
If $\omega$ is a differential $k$-form and $Y$ is a vector field, we denote by $i_Y \omega$  the $(k-1)$-form 
$$ (i_Y\omega)(X_1,\ldots, X_{k-1})=\omega(Y, X_1,\ldots, X_{k-1}).$$
Let $E(x)=x$ denote the Euler vector field on $V$ and let  $\pi\colon V\setminus\{0\}\to \PP_+(V)$ denote the canonical projection. The proof of the following lemma is straightforward and therefore omitted.
\begin{lemma}\label{lem:P+} Let $\omega\in\largewedge^n V^*$ 
be  a fixed non-zero constant differential $n$-form on $V$ and let $f\colon V\setminus\{0\}\to\CC$ be continuous and homogeneous of degree $-n$. Then the following assertions hold: 
	\begin{enuma} \item $f  i_E \omega = \pi^* \omega_f $ 
for  a unique differential $(n-1)$-form $\omega_f$ on $\PP_+(V)$. 
	\item $ T^*\omega_f =  \det T \cdot \omega_{f\circ T} $ for all $T\in \GL(V)$. 
	\end{enuma}
\end{lemma}

Recall that $\PP_V = V\times \PP_+(V^*)$. Let $\pi_1\colon \PP_V\to V$ and $\pi_2 \colon \PP_V\to \PP_+(V^*)$ denote the projections to the first and second factor of $\PP_V$. Let $U_0\subset \PP_V$ denote the open subset of all $(x,[\xi])$ with $x\neq 0$. 

\begin{definition}
	Let  
	
	$\vol=\epsilon\otimes \omega\in D(V)\simeq or(V)\otimes \largewedge^n V^*$ be a positive density and $p,q$ be non-negative integers. Let $\phi\in \Sym^pV^* \otimes \Sym^q V$ and $\psi\in \Sym^p V \otimes \Sym^q V^*$. Define two smooth differential $(n-1)$-forms 
	$\omega_1(\phi), \omega_2(\psi)\in or(V)\otimes \Omega^{n-1}(U_0)$ with values in $or(V)$	
	by 
	\begin{align*}
		\omega_1(\phi)|_{(x,[\xi] )} &= \epsilon\otimes  |\langle \phi, x^p \otimes \xi^q \rangle| \langle x,\xi\rangle^{-q} \pi_1^*( i_x \omega)\\
		\omega_2(\psi) |_{(x,[\xi])} &= \epsilon\otimes  |\langle \psi,  \xi ^p \otimes x^q     \rangle| \langle x,\xi\rangle^{-p} \pi_2^*( \langle x,\xi\rangle^{-n} i_\xi \omega^*),
	\end{align*}
where  $\omega^*\in \largewedge^n V$ satisfies $\langle \omega,\omega^*\rangle =1$. 
\end{definition}

\begin{lemma} \label{lemma:invariance}
	The forms $\omega_1(\phi), \omega_2(\psi)$ are well-defined and satisfy 
	$$ T^* \omega_1(\phi) = \omega_1(T^*\phi) \quad{ and }\quad  T^* \omega_2(\psi) = \omega_2(T^*\psi)$$
	 for the natural action  $T(x,[\xi])= (Tx, [T^{-*} \xi])$ of $\SL(V)$ on $\PP_V$.  
\end{lemma}
\begin{proof}
The  well-definedness and invariance of $\omega_2(\psi)$ follows from Lemma~\ref{lem:P+}. The rest is straightforward.
\end{proof}

Observe that after a choice of euclidean inner product on $V\simeq \RR^n$ with volume form 
$\vol=\omega$  we have $\PP_V = \RR^n\times S^{n-1}$ and
	\begin{align*}
	\omega_1(\phi)|_{(x,u)} &=  |\langle \phi, x^p \otimes u^q \rangle| \langle x,u\rangle^{-q} \pi_1^*( i_x \vol),\\
	\omega_2(\psi) |_{(x,u)} &=  |\langle \psi,  u^p \otimes x^q     \rangle| \langle x,u\rangle^{-(n+p)} \pi_2^*( i_u \vol).
\end{align*}

\begin{proof}[Proof of Theorem~\ref{thm:newMink}]
If $K\subset V$ is a convex body containing the origin in the interior, then the 
support of $\nc(K)$ is contained in $\pi^{-1}(\partial K)\subset U_0$. Hence using a suitable cut-off function and Proposition~\ref{prop:nc}, we see that  
$ \int_{\nc(K)} \omega_1(\phi)$ and $\int_{\nc(K)} \omega_2(\psi)$ define
  continuous valuations on $\calK_0(V)$. 

We claim that there exists a convex body $\Phi^{p,q} K\subset \Sym^pV\otimes \Sym^q V^*$ such that 
\begin{equation}\label{eq:convexity} h_{\Phi^{p,q} K}(\phi) = \int_{\nc(K)} \omega_1(\phi)\end{equation}
holds for all $\phi\in \Sym^pV^* \otimes \Sym^q V$.
Clearly,  the right-hand side of \eqref{eq:convexity} is a $1$-homogenous function of $\phi$. 
Let us fix a euclidean inner product on $V\simeq \RR^n$ with volume form $\vol=\omega$. Let $K\subset \RR^n$ be a convex body with smooth and strictly positively curved boundary.  Then $\b \nu\colon \partial K\to \nc(K)$, $\b\nu(x)=(x,\nu_K(x))$, is a diffeomorphism and   $\b\nu ^* \pi_1^*( i_x \vol) = \langle x, \nu(x) \rangle  \vol_{\partial K}$, where  $\vol_{\partial K}$ denotes the riemannian volume form. 
This proves  that 
$$  \int_{\nc(K)} \omega_1(\phi) = 
\int_{\partial K }  |\langle \phi, x^p \otimes \nu(x)^q\rangle | \langle x, \nu(x)\rangle^{1-q} dx$$
 is a sublinear function of $\phi$. By continuity, it follows that the right-hand side of \eqref{eq:convexity} is sublinear for all convex bodies $K\in \calK_0(\RR^n)$. 

An analogous argument using $ \b\nu ^* \pi_2^*( i_u \vol) = \kappa_K(x)  \vol_{\partial K}$
shows the existence of a convex body $\Psi^{p,q}(K)$ in $\Sym^p V^* \otimes \Sym^q V$ such that 
$$h_{\Psi^{p,q} K}(\psi) = \int_{\nc(K)} \omega_2(\psi).$$

Finally, $\SL(V)$-invariance follows from Lemma~\ref{lemma:invariance}. 	
	
\end{proof}

Let us write $\Phi_{V}^{p,q}$ and $\Psi_{V}^{p,q}$ to emphasize the domain of definition. Since there is a canonical density on $V\times V^*$, namely the symplectic volume, we obtain a natural isomorphism $D(V^*)\simeq D(V)^*$. 

\begin{proposition}\label{prop:polarityPhiPsi}
Let 
 $\vol\in D(V)$ be a positive density on $V$ and let $\vol^*\in D(V)^*\simeq D(V^*)$ be defined by $\langle \vol, \vol^*\rangle =1$. 
After the natural identification $V^{**}\simeq V$,  
 the  relations $$ \Phi_{V^*}^{p,q}(K^* ) = \Psi_{V}^{p,q}(K)\quad \text{and}\quad \Psi_{V^*}^{p,q}(K^* ) = \Phi_{V}^{p,q}(K)$$
hold for all $K\in \calK_0(V)$.	
\end{proposition}
\begin{proof}
 Let $\psi \in \Sym^p V \otimes \Sym^q V^*$ and use the dual  density
 	$\vol^*$ to define $\omega_{1,V^*}^{p,q}$.   Let $F_{V}$ denote the map of Lemma~\ref{lem:polar}. Then 
$$ (F_V)^*  \omega_{1,V^*}^{p,q}(\psi) = \omega_{2,V}^{p,q}(\psi)$$
and the first identity follows from Lemma~\ref{lem:polar}. The same argument proves the second identity. 

\end{proof}

\subsection{$L^p$ projection and centroid bodies} \label{sec:Lp}
Before we proceed to construct  further examples of  affine Minkowski valuations from the families $\Phi^{p,q}$ and $\Psi^{p,q}$  , let us pause for a moment to take a closer look at the boundary cases $p=0$ and $q=0$.

Let us fix a euclidean inner product on $V\simeq \RR^n$ with  euclidean density $\vol$. 
For $p=0$ it follows from  \eqref{eq:PhipqK} that 
$$ h_{\Phi^{0,q}K}(\xi )= \int_{S^{n-1}} |\langle \xi , u^q\rangle | h_K(u)^{1-q}   dS_{n-1}(K,u).$$
In particular, for $(p,q)=(0,1)$ we recover the  classical projection body of $K$. 
Recall that for $q\geq 1$ the $L^q$ projection body of $K\in \calK_0(\RR^n)$ introduced by Lutwak, Yang, and Zhang \cite{LYZ:LpAffine}  is  up to normalization the convex body with support function 
$$ h_{\Pi_q K}(v) = \left( q!\int_{S^{n-1}} |\langle u,v\rangle|^q h_K(u)^{1-q} dS_{n-1}(K,u)\right)^{1/q}.$$
Hence, if $q$ is an integer, then it follows from \eqref{eq:innerprodTensor} that
$$ h_{\Phi^{0,q}K}(v^q)=  h_{\Pi_q K}(v)^q, \quad v\in S^{n-1}.$$

Lutwak, Yang, and Zhang \cite{LYZ:LpAffine}  define for $p\geq 1$ the  $L^p$ centroid body of  the convex body $K$  up to normalization  as the convex body with support function
$$ h_{\Gamma_p K}(v)= \left(p!(n+p)\int_K |\langle v, x\rangle|^p dx\right)^{1/p}.$$
To see the connection with the family $\Phi^{p,q}$, we use  the following simple consequence of the divergence theorem.

\begin{lemma} \label{lemma:lambdaHom}Let $\lambda>-n$ be a real number  and let  $f\colon \RR^n\setminus \{0\} \to\RR$ be continuous and $\lambda$-homogeneous. Then  for every convex body $K$ containing the origin in its interior
	$$ \int_{\partial K} f(x) \langle x, \nu(x)\rangle dx = (\lambda + n)  \int_K f(x) dx.$$
	\end{lemma}

Applying Lemma~\ref{lemma:lambdaHom} to \eqref{eq:PhipqK} and using \eqref{eq:innerprodTensor}, we obtain $$ h_{\Phi^{p,0}K}(v^p)  =  h_{\Gamma_p K}(v)^p,\quad v\in S^{n-1},$$
for every integer $p\geq 1$ and every convex body $K$ containing the origin in its interior.

\subsection{Further  examples of affine Minkowski valuations}

The  $\mathrm{SL}(V)$-representation $\Sym^pV \otimes \Sym^q V^*$ is not irreducible for $p,q>0$ and $n>1$. Therefore it decomposes into a direct sum of isotypic components. 
 If $\pi_\lambda$ denotes the projection of $\Sym^pV \otimes \Sym^q V^*$ onto the isotypic component corresponding to the highest weight $\lambda$, then
$$ \Phi_{\lambda}^{p,q} = \pi_\lambda \circ \Phi^{p,q} \colon \calK_{0}(V) \to \calK(W_\lambda) $$
defines  a continuous $\mathrm{SL}(V)$- invariant Minkowski valuation; analogously $\Psi_\lambda^{p,q}$ is defined.
Let us first show that the resulting Minkowski valuations are indeed non-trivial.
\begin{lemma}\label{lemma:nontrivial}

	$\Phi^{p,q}_{\lambda}$ and $\Psi^{p,q}_{\lambda}$ are non-trivial.  
\end{lemma}
\begin{proof} It suffices to consider $\Phi^{p,q}_{\lambda}$. 
	Let $\phi \in \Sym^p V^* \otimes \Sym^qV$. We may think of $\phi=\phi(x,\xi)$ as a polynomial on $V\times V^*$.  
	If $\phi$ vanishes on the open set $\langle x, \xi\rangle >0$, then $\phi$ vanishes identically. Hence, if $\phi\neq 0$, then there exists $(x,\xi)$ satisfying $\langle x, \xi \rangle >0$ such that $\phi(x,\xi)\neq 0$. Clearly there exists a smooth convex body $K$ containing the origin in the interior such that $x$ is a boundary point of $K$ and $\xi$ an outward-pointing conormal to $K$ at $x$. In fact, $K$ can be chosen to be an ellipsoid.
	We conclude that 
	$h_{\Phi^{p,q}(K)}(\phi)>0$.  
\end{proof}

It remains to determine the highest weights appearing in the isotopic decomposition and to determine their multiplicity.

\begin{lemma}
	\label{lemma: decomposition of Symp tensor Symq}
	The highest weights appearing  in the isotypic  decomposition of the $\SL(V)$-representation  $\Sym^pV \otimes \Sym^q V^*$ are of the form
	\begin{equation}\label{eq:hwt} (p+q-2i)\varepsilon_1 + (q-i) (\varepsilon_2 + \dots + \varepsilon_{n-1}), \quad i=0,\dots , \min\{p,q\}. \end{equation}
	The second summand has to be interpreted as zero  if $n=2$.  Moreover, each such highest weight appears with multiplicity one. 
\end{lemma}

\begin{proof}
This follows from the correspondence between representations of $\SL(V)$ and $\GL(V)$ and the decomposition of $\Sym^pV \otimes \Sym^q V^*$ with respect to the latter group, see Lemma~\ref{lemma:decompGL} and Section~\ref{sec:SL}.
\end{proof}

Theorem~\ref{thm:TensorDecomp} is now an immediate consequence of  Lemma~\ref{lemma: decomposition of Symp tensor Symq} and Lemma~\ref{lemma:nontrivial}. Since any finite-dimensional $\mathrm{SL}_n(\RR)$-representation is uniquely determined by a partition of size $n-1$, also Corollary~\ref{cor:dimLeq3} follows at once.

\subsection{A closer look at the case $(p,q)=(1,1)$}
\label{sec:11}
For $p,q>0$ the Minkowski valuations $\Phi^{p,q}$ display  features distinctively different from those of $L^p$ projection and centroid bodies. Let us illustrate this by taking a closer look at the simplest case $(p,q)=(1,1)$. The $\GL(V)$-representation $V\otimes V^*\simeq \End(V)$ is reducible, since there is the invariant subspace of trace-free endomorphisms. Let $W_0$ denote the kernel of the trace map. Then $W_0$ is irreducible with highest weight $\lambda= \varepsilon_1-\varepsilon_n$, and 
$$ \End(V) = W_0 \oplus \langle I \rangle .$$  

Choose a euclidean inner product on $V$ with  euclidean density $\vol$. Then by \eqref{eq:PhipqK} for every $A\in \End(\RR^n)$, we have 
\begin{equation}\label{eq:11}  h_{\Phi^{1,1} K} (A) = \int_{\partial K} | \langle Ax,\nu_K(x)\rangle|   dx.\end{equation}
Consequently, 
$$ h_{\Phi^{1,1} K} (I)= \int_{\partial K}  \langle x,\nu_K(x)\rangle   dx = n\vol_n(K),$$
as expected, since this corresponds to the projection of $\End(V)$ onto the trivial subrepresentation $\langle I\rangle$. 

A more interesting consequence of \eqref{eq:11} is that if $A=-A^t$ is skew-symmetric and $B^n\subset \RR^n$ denotes the euclidean unit ball, then 
$h_{\Phi^{1,1} B^n} (A) =0$.  Thus the image of the unit ball under $\Phi^{1,1}$ is not full-dimensional, but contained in the subspace of symmetric matrices. By $\SL(V)$-invariance, the image of every origin-symmetric ellipsoid is not full-dimensional.

Let $S\subset \RR^n$ be any simplex containing the origin in its interior, let $S_1,\ldots, S_{n+1}$ be its facets, and let $u_1,\ldots, u_{n+1}$ denote  the corresponding unit normals. By \eqref{eq:convexity}, we have 
$$  h_{\Phi^{1,1} S} (A) =\sum_{i=1}^{n+1}  \int_{S_i} | \langle Ay,u_i\rangle| dy.$$  
If $A\neq 0$, then there exists a normal $u_i$ such that the linear function $x\mapsto \langle Ax,u_i\rangle $ on $S_i$ is non-zero and therefore $ h_{\Phi^{1,1} S} (A)>0$. It follows that the image under $\Phi^{1,1}$ of every simplex containing the origin in its interior is a full-dimensional convex body in $\End(V)$. 

Note that $\Phi^{1,1}$ is homogeneous of degree $n$. By Lemma~\ref{lemma:imageSL}, the ratio 
$$\frac{\vol_{n^2}(\Phi^{1,1}K·)}{\vol_n(K)^{n^2}}$$
is invariant under the action of $\GL(V)$. It vanishes on origin-symmetric ellipsoids, but is positive on simplices containing the origin.

\section{A family of upper semicontinuous Minkowski valuations}

In this section we construct a family of $\SL(V)$-invariant Minkowski valuations,  closely related to the affine surface area of a convex body, that are not continuous, but merely upper semicontinuous. 
Property \eqref{eq:upper semi}, and hence the upper semicontinuity of a Minkowski valuation, can be conveniently rephrased in terms of support functions.
\begin{lemma} 
	For every sequence of convex bodies $(L_j)$ in $W$ and every convex body $L$ in $W$, the following statements are equivalent: 
	\begin{enuma}
		\item $\limsup_{j\to \infty} h_{L_j}(\xi) \leq h_L(\xi)$ holds for every $\xi \in W^*$. 
		\item For every convex body $D$ containing the origin in its interior  there exists $j_0\in \NN$ such that 
		$$ L_j \subset L+  D\quad \text{for all } j\geq j_0.$$
	\end{enuma}
\end{lemma}

Let us also acknowledge the following elementary fact about the class $\mathrm{Conc}(0,\infty)$. 

\begin{lemma}
	If $f\in \mathrm{Conc}(0,\infty)$, then $f$  is monotonically non-decreasing and $f(t)/t$ is monotonically decreasing.
\end{lemma}

We call a subset $C\subset \RR^n$ conical if $x\in C$ and $t>0$ imply $tx\in C$.

\begin{proposition}\label{prop:LebesgueCK}
Let $\phi\in \Sym^p (\RR^n)^*$ be fixed. For every convex body  $K\in\calK_0(\RR^n)$ define two Borel measures on $\RR^n$ by 
 \begin{align*}
 	\mu_K(U) & = \int_{\partial K \cap U}  |\langle \phi, x^p\rangle| \langle x,\nu_K(x)\rangle  dx,\\
 	 C_K(U)  &=  \int_{\nc(K)\cap \pi_1^{-1}(U)}  |\langle \phi, x^p\rangle | \langle x,u\rangle^{-n} \pi_2^*\omega_{S^{n-1}},
 \end{align*}
where $\omega_{S^{n-1}}$ denotes the Riemannian volume form on the sphere $S^{n-1}$.  The following properties hold:
\begin{enuma}
\item  For every conical Borel set $U\subset \RR^n$ and every convergent sequence $K_j\to K$ in $\calK_0(\RR^n)$ 
$$ \lim_{j\to \infty} \mu_{K_j} (U) = \mu_K (U).$$
\item  For every closed set $A\subset \RR^n$ and every convergent sequence $K_j\to K$ in $\calK_0(\RR^n)$
 $$ \limsup_{j\to \infty} C_{K_j}(A) \leq C_{K}(A).$$

\item 
Let $$C_K \mres \partial K = C^a_K + C^s_K$$
denote the Lebesgue decomposition  of $C_K \mres \partial K$ with respect to the restriction of the $(n-1)$-dimensional Hausdorff measure to $\partial  K$. 
Then  $C^a_K$ is concentrated on  $(\partial K)_+ $ and 
$$ C^a_K(U) = \int_{\partial K \cap U } \wt\kappa_K(x) d\mu_K(x)$$
for every Borel set $U\subset \RR^n$. 
\end{enuma}

\end{proposition}

\begin{proof}
(a) If $h\colon \RR^n\setminus \{0\} \to \RR$ is smooth and    homogeneous of degree $\lambda> -n$,  
the divergence theorem and passing to polar coordinates imply 
$$ \int_{\partial K} h(x) \langle x, \nu_K(x)\rangle  dx 
= \int_{S^{n-1}} h(u) \rho_K(u)^{\lambda+n} du.$$
By approximation, this identity continues to hold for  Borel functions $h$ that are homogeneous of degree  $\lambda> -n$ and bounded when restricted to $S^{n-1}$. 
The continuous dependence on $K$ of the right-hand side is clear by dominated convergence.

(b) Using a smooth cut-off function  around the origin, we can replace $ \langle x,u\rangle^{-n} \pi_2^*\omega_{S^{n-1}}$ by a smooth differential form without changing the integral. 
Hence, Proposition~\ref{prop:nc} implies $C_{K_j}\to C_{K}$ weakly as measures and the claim follows from the Portmanteau theorem.

(c)
Lemma~\ref{lemma:Gauss Lipschitz} implies  that the graphing map $\b \nu_K \colon (\partial K)_r \to \nc(K)\subset S\RR^n$, $\b\nu_K(x)= (x,\nu_K(x))$, is Lipschitz. This shows that the absolutely continuous part of $C(K,\cdot)\mres \partial K$ is concentrated on $(\partial K)_+$.
The expression for the absolutely continuous part follows from the area formula and  \cite[eq. (4)]{Hug:Contributions}.

\end{proof}

\begin{lemma}\label{lemma:Lusin} Let $\phi\in \Sym^p (\RR^n)^*$ be fixed. For every $\epsilon>0$ there exists a  closed subset $\omega$ of $ \partial K$ satisfying the properties
	\begin{enuma}
		\item $\omega\subset (\partial K)_+$
		\item $\wt \kappa_K$ restricted to $\omega$ is continuous
		\item  $\mu_K( K \setminus \omega)< \epsilon$. 
	\end{enuma}
\end{lemma}
\begin{proof}
	
 	By Lusin's theorem \cite[Theorem 1.2.2]{EvansGariepy} there exists a compact  subset $\omega\subset (\partial K)_+$ such that the restriction of $\wt \kappa_K$ to $\omega$ is continuous and  $\mu_{K}((\partial K)_+\setminus \omega)$ has arbitrarily small measure. 
\end{proof}

\begin{proof}[Proof of Theorem~\ref{thm:semi}]
 We will establish the properties of $\Phi_f^p$ first and will deduce those of $\Psi_f^p$ via polarity. 
 Let us write $\Phi=\Phi_f^p$ for brevity. Moreover, fix a euclidean inner product on $V\simeq \RR^n$ with  euclidean density $\vol$. 
 	
As a  consequnce of the concavity of $f$, Jensen's inequality, and Proposition~\ref{prop:LebesgueCK}, item (c), we obatin
	\begin{align*} \frac{1}{ \mu_{K}(\partial K)} \int_{\partial K}  f(\wt \kappa_K(x)) d\mu_{K}(x)& \leq  f(\frac{1}{ \mu_{K}(\partial K)} \int_{\partial K}  \wt \kappa_K(x) d\mu_{K}(x) )\\
	& = f( \frac{C_{K}^a(\partial K)}{ \mu_{K}(\partial K)})  
\end{align*} 	 
Hence the  integral \eqref{eq:Phifp}   is always finite. 
	Since  the right-hand side of \eqref{eq:Phifp} is clearly positively homogeneous and subadditive in $\phi$,  we have defined a map $\Phi \colon \calK_0(\RR^n)\to \calK(\Sym^p \RR^n)$. It follows from Section~\ref{sec:GaussAffine} that this map is $\SL_n(\RR)$-equivariant.

 Repeating Sch\"utt's proof \cite{Schutt:Asa} of the valuation property of the affine surface area shows that $\Phi$ is a valuation (see also \cite{Ludwig:GeneralAsa}).

It remains to prove the  upper semicontinuity of $\Phi$. Following closely Ludwig's argument \cite{Ludwig:Semicontinuity}, let us   fix a non-zero $\phi\in\Sym^p(\RR^n)^*$ and a convex body  $K\in \calK_0(\RR^n)$. Choose $\epsilon, \eta >0$.
Let $\omega\subset \partial K$ have the properties specified in Lemma~\ref{lemma:Lusin}. 

Set $$a= \min_{x\in \omega} \wt \kappa_K(x)\quad \text{and}\quad b = \max_{x\in \omega} \wt \kappa_K(x).$$
Since $f$ is uniformly  continuous on $[a,b]$, there exists $\delta>0$ such that 
 $|s-t|< \delta $ implies $|f(s)-f(t)|<\eta$.

Denote by $\wt \omega=\{ tx\colon x\in \omega, t\geq 0\}$ the cone spanned by $\omega$. 
Choose a subdivision  $a=t_0<t_1< \cdots < t_{m+1}=b$ of $[a,b]$ such that $t_{j+1}-t_j< \delta$ and such that 
$$ \calH^{d-1}(\{x\in \partial K_j \cap \wt \omega \colon  \wt\kappa_{K_j}(x)=t_i\})=0,$$ where $\calH^{d-1}$ denotes the $(d-1)$-dimensional Hausdorff measure, 
for all $i$ and $j$. 

Put 
$$ \omega_i = \{ x\in  \omega \colon  t_{i}\leq  \wt \kappa_{K}(x)\leq t_{i+1}\}$$
and let $\wt \omega_i$ denote the cone spanned by $\omega_i$. 

The monotonicity of $f$ obviously implies  
$$ \int_\omega f(\wt\kappa_K(x)) d\mu_K(x) = \sum_{i=1}^m 
\int_{\omega_i} f(\wt\kappa_K(x)) d\mu_K(x) \geq   \sum_{i=1}^m f(t_i) \mu_K(\omega_i).$$

It follows from Proposition~\ref{prop:LebesgueCK} that
$$ C_{K}( \omega_i) = C^a_{K}(\omega_i)\leq t_{i+1}  \mu_{K}(  \omega_i)$$
and 
$$ \int_{\partial K_j \cap \wt \omega_i}  \wt\kappa_{K_j}(x) d\mu_{K_j}(x)  \leq C_{K_j}( \wt \omega_i) .$$ 
Jensen's inequality and the latter inequality imply
 \begin{align*} \int_{\partial K_j \cap \wt \omega}  f(\wt\kappa_{K_j}(x)) d\mu_{K_j}(x)
 	&  = \sum_{i=1}^m  \int_{\partial K_j \cap \wt \omega_i}  f(\wt\kappa_{K_j}(x)) d\mu_{K_j}(x)\\
 	 	&  \leq \sum_{i=1}^m  f\left(\frac{1}{\mu_{K_j}(\wt \omega_i)} \int_{\partial K_j \cap \wt \omega_i}  \wt\kappa_{K_j}(x) d\mu_{K_j}(x)\right) \mu_{K_j}(\wt \omega_i)\\
 	 	& \leq  \sum_{i=1}^m  f(\frac{C_{K_j}(\wt \omega_i)}{\mu_{K_j}(\wt \omega_i)}) \mu_{K_j}(\wt \omega_i).
 \end{align*}
Since $\wt \omega_i$ is closed, Proposition~\ref{prop:LebesgueCK}, item (b), implies  $ \limsup_{j\to \infty} C_{K_j}(\wt \omega_i) \leq C_{K}(\wt \omega_i) $
while  item (a) of that proposition guarantees $ \lim_{j\to \infty} \mu_{K_j}(\wt \omega_i) = \mu_{K}(\wt \omega_i)$. We conclude that 
\begin{align*}  \limsup_{j\to \infty}  \int_{\partial K_j \cap \wt \omega}  f(\wt\kappa_{K_j}(x)) d\mu_{K_j}(x) &  \leq  \limsup_{j\to \infty}
	 \sum_{i=1}^m  f(\frac{C_{K_j}(\wt \omega_i)}{\mu_{K_j}(\wt \omega_i)}) \mu_{K_j}(\wt \omega_i)\\
	 & \leq  \sum_{i=1}^m  f(\frac{C_{K}( \omega_i)}{\mu_{K}( \omega_i)}) \mu_{K}(\omega_i)\\
	 & \leq  \sum_{i=1}^m  f(t_{i+1}) \mu_{K}(\omega_i)\\
	 & \leq  \sum_{i=1}^m  f(t_{i}) \mu_{K}(\omega_i) + \eta \mu_K(\omega)\\
	 & \leq \int_\omega f(\wt\kappa_K(x)) d\mu_K(x) + \eta \mu_K(\omega).
	\end{align*}
Letting  $\eta \to 0$, we obtain
$$ \limsup_{j\to \infty}  \int_{\partial K_j \cap \wt \omega}  f(\wt\kappa_{K_j}(x)) d\mu_{K_j}(x) \leq  \int_{\partial K} f(\wt\kappa_K(x)) d\mu_K(x) .$$

Since $f(t)$ is non-decreasing and $f(t)/t$ is non-increasing, we have  for every $t>0$
\begin{align*} 
	\int_{\partial K_j \setminus  \wt \omega}  f(\wt\kappa_{K_j}(x)) d\mu_{K_j}(x)  &= \int_{\{ \wt\kappa_{K_j}\leq t\} \setminus  \wt \omega } 
		f(\wt\kappa_{K_j}(x)) d\mu_{K_j}(x)\\
		&  +  \int_{\{ \wt\kappa_{K_j}> t\} \setminus  \wt \omega} 
			\frac{f(\wt\kappa_{K_j}(x))}{\wt\kappa_{K_j}(x)} \wt\kappa_{K_j}(x) d\mu_{K_j}(x)\\
			& \leq f(t) \mu_{K_j}(\RR^n\setminus \wt \omega) + 
			\frac{f(t)}{t} C_{K_j}(\RR^n).
\end{align*}
Hence 
$$ \limsup_{j\to \infty}  \int_{\partial K_j}  f(\wt\kappa_{K_j}(x)) d\mu_{K_j}(x) \leq  \int_\omega f(\wt\kappa_K(x)) d\mu_K(x)  +f(t)\epsilon  +\frac{f(t)}{t}C_{K}(\RR^n) .$$
Sending in this inequality  first $\epsilon\to 0$ and then $t\to \infty$, the claim follows.
\end{proof}

\begin{lemma}\label{lemma:polaritySemi}
	If we identify $(\RR^n)^*\simeq \RR^n$ via the euclidean inner product, then  
		\begin{align*}h_{\Phi_f^p (K^*)}(\phi)& = \int_{S^{n-1}}  |\langle \phi, u^p\rangle| h_{K}(u)^{-(n+p)}  f(\wt\alpha_K(u)) du\\
			h_{\Psi_f^p (K^*)}(\psi)& = \int_{S^{n-1}}  |\langle \psi, \nabla h_K(u) ^p\rangle | h_K(u)^{-n} f(\wt\alpha_K(u)) du
		\end{align*}
		holds for every convex body $K\in \calK_0(\RR^n)$,  $f\in \mathrm{Conc}(0,\infty)$ and $\phi,\psi\in\Sym^p(\RR^n)$.
	\end{lemma}
	\begin{proof}
	Set $h=h_K$ for the sake of brevity. 
	We apply the area formula to the Lipschitz map $F\colon S^{n-1}\to \partial K^*$, $F(u)= h(u)^{-1} u$.
	The Jacobian of $F$ is 
	$$ JF(u)=h(u)^{-n} |\nabla h(u)|. $$  Note that 
	$$ \nu_{K^*}(F(u)) = |\nabla h(u)|^{-1}\nabla h(u) 	$$
	 if $h$ is differentiable at $u$
	and hence
	$$\langle F(u), \nu_{K^*}(F(u))\rangle = |\nabla h(u)|^{-1}.$$
	Since by Corollary~\ref{cor:kappaAlpha}
	$$  \wt \kappa_{K^*}(F(u))= \wt\alpha_K(F(u))= \wt\alpha_K(u),$$
	the above yields
	\begin{align*}h_{\Phi_f^p (K^*)}(\phi)& =  \int_{\partial K^*} |\langle \phi, x^p\rangle| \langle x, \nu_{K^*}(x)\rangle f(\wt \kappa_{K^*}(x)) dx \\
		& = \int_{S^{n-1}}  |\langle \phi, u^p\rangle| h_{K}(u)^{-p} |\nabla h_K(u)|^{-1} f(\wt\alpha_K(u) ) JF(u) du\\
	& =  \int_{S^{n-1}}  |\langle \phi, u^p\rangle| h_{K}(u)^{-(p+n)} f(\wt\alpha_K(u) ) du.	
	\end{align*}
	\end{proof}

For $f\in\mathrm{Conc}(0,\infty)$ define  $f^*(t)= tf(1/t)$. Note that $f^*\in \mathrm{Conc}(0,\infty)$ and  $(f^*)^*=f$. 
One may replace integration over the boundary of $K$  by integration over the unit sphere as follows.

\begin{lemma} \label{lemma:sphere} For every $K\in\calK_0(\RR^n)$ and $\phi,\psi\in\Sym^p(\RR^n)$
	\begin{align*} h_{\Phi_f^p(K)}(\phi) &= \int_{S^{n-1}}  |\langle \phi, \nabla h_K(u) ^p\rangle | h_K(u)^{-n} f^*(\wt\alpha_K(u)) du,\\
		h_{\Psi_f^p(K)}(\psi)& = \int_{S^{n-1}}  |\langle \psi, u^p\rangle | h_K(u)^{-(n+p)} f^*(\wt\alpha_K(u)) du.
	\end{align*}
\end{lemma}

\begin{proof}
	This follows from the area formula applied to the Lipschitz map $\nu_K\colon (\partial K)_r\to S^{n-1}$. Let us write $\wt \kappa = \wt \kappa_K$, $\wt \alpha= \wt \alpha_K$, and $h=h_K$ for brevity.
	We have $J\nu_K(x) = \kappa(x)$. Since   $h$ is second order differentiable at a.e.\ $u\in \nu_K((\partial K_r))$ and 
	$\kappa(\nabla h(u))\alpha(u)=1$ at those points by Lemma~\ref{lemma:alphakappa},   the area formula yields 
	\begin{align*}
		\int_{(\partial K)_r}  |\langle \phi, x^p\rangle |& \langle x, \nu_K(x)\rangle f(\wt\kappa(x)) dx\\
		&=  \int_{\nu_K((\partial K)_r) }  |\langle \phi, \nabla h(u)^p\rangle | \langle \nabla h(u),u \rangle  f(\wt\alpha(u)) \alpha(u) du \\ 
		& =\int_{\nu_K((\partial K)_r) }  |\langle \phi, \nabla h(u)^p\rangle | h(u)^{-n} f^*(\wt\alpha(u)) du.
	\end{align*}
Letting $r\to 0$, 
	by monotone convergence and the fact that $x\in (\partial K)_+$ if $\alpha(u)>0$, $u=\nu_K(x)$, by Lemma~\ref{lemma:K+},  the claim follows. 
	Similar reasoning proves the statement for $\Psi_f^p$. 
\end{proof}

Let us write $\Phi_{f,V}^p$ and $\Psi_{f,V}^p$ to emphasize the dependence on $V$.  The following  should be compared with analogous statements  for general affine surface areas in \cite[Theorem 4]{Ludwig:GeneralAsa} and \cite{Hug:CurvRel}.

\begin{proposition}
\label{prop:polarity}
Let $\vol$ be a positive density on $V$ and let $\vol^*\in D(V)^*\simeq D(V^*)$ be defined by $\langle \vol, \vol^*\rangle =1$. 
After the natural identification $V^{**}\simeq V$, 
$$\Phi_{f,p,V^*}(K^*) = \Psi_{f^*,p,V}(K) \quad \text{and}\quad \Psi_{f,p,V^*}(K^*)= \Phi_{f^*,p,V}(K)$$
holds for every convex body $K\in \calK_0(V)$. 
\end{proposition}
\begin{proof}

Choosing a euclidean inner product on $V$ with euclidean density  $\vol$, we have $V\simeq \RR^n \simeq V^*$ and $\vol^*=\vol$. The claim now follows from Lemmas~\ref{lemma:polaritySemi} and \ref{lemma:sphere}.

\end{proof}

\section{Open questions}

A basic question that remains open  is whether the examples of  invariant continuous  Minkowski valuations for irreducible representations constructed in this paper essentially exhaust already all possibilities. More precisely: 

\begin{question}
Is the condition on  the highest weight of $W$ specified in Theorem~\ref{thm:TensorDecomp} not only a sufficient condition for the existence of a non-trivial invariant continuous Minkowski valuation $\calK_0(V)\to\calK(W)$, but also necessary?
\end{question}

The following variant of the question seems also of interest: 

\begin{question}Do upper semicontinuous valuations arise only if $W\simeq \Sym^p V$ or $W\simeq \Sym^p V^*$? In particular, do there exist upper semicontinuous, but not continuous, invariant Minkowski valuations if $W = V\otimes V^*$?
\end{question}
 There are candidates for such  Minkowski valuations, but it  is not clear that they are upper semicontinuous.

As discussed in Section~\ref{sec:Lp}, the Minkowski valuations $\Phi^{p,0}$ and $\Phi^{0,q}$ are closely related to the $L^p$ centroid and $L^p$ projection bodies. Lutwak, Yang, and Zhang \cite{LYZ:LpAffine} have  established sharp affine isoperimetric inequalities for $L^p$ projection and $L^p$ centroid bodies. 
Let $d = \binom{n+p-1}{p} \binom{n+q-1}{q}$. It follows from Lemma~\ref{lemma:imageSL} and the homogeneity of $\Phi^{p,q}$  that there  exists a number $r$ such that the ratio 
$$\frac{\vol_d(\Phi^{p,q}K )}{ \vol_n(K)^r}$$
	 is invariant under the action of 
$\GL_n(\RR)$. One may ask for which convex bodies $K$ containing the origin in the interior this ratio is minimized or maximized. As the example discussed in Section~\ref{sec:11} shows, this minimum may be zero.
A comparison with the situation for the classical $L^p$ centroid bodies suggests that the following question might be a good starting point.

\begin{question} Let $p\geq 2$, $d=\binom{n+p-1}{p}$, and $r=d(p+n)/n$.
For which convex bodies $K$ in $\RR^n$ containing the origin in their interior
is  the ratio 
$$\frac{\vol_d(\Phi^{p,0}K )}{ \vol_n(K)^{r}}$$
minimal?
\end{question}

\bibliographystyle{abbrv}
\bibliography{ref_papers,ref_books}

\end{document}